\documentclass[11pt]{article}

\usepackage{amsmath}
\usepackage{amsfonts}
\usepackage{amssymb}
\usepackage{amsthm}
\usepackage{graphicx}
\usepackage{color}
\usepackage{enumitem}
\usepackage{float}
\usepackage{hyperref}
\usepackage[top=1.2in,bottom=1.2in,right=1in,left=1in]{geometry}
\usepackage[numbers,square,sort&compress,longnamesfirst]{natbib}
\usepackage{endnotes}
\let\footnote=\endnote

\theoremstyle{plain}
\newtheorem{theorem}{Theorem}[section]
\newtheorem{corollary}[theorem]{Corollary}

\newtheorem{definition}[theorem]{Definition}

\newtheorem{lemma}[theorem]{Lemma}
\newtheorem{proposition}[theorem]{Proposition}
\newtheorem{remark}[theorem]{Remark}

\numberwithin{equation}{section}

\newcommand{\tr}{\operatorname{Tr}}

\newcommand{\R}{\mathbb{R}}

\def\og{\leavevmode\raise.3ex\hbox{$\scriptscriptstyle\langle\!\langle$~}}
\def\fg{\leavevmode\raise.3ex\hbox{~$\!\scriptscriptstyle\,\rangle\!\rangle$}}

\begin{document}

\author{Eberhard Mayerhofer\thanks{Department of Mathematics \& Statistics, University of Limerick, Limerick, V94 T9PX, Ireland; \url{eberhard.mayerhofer@ul.ie }}, Robert Stelzer\thanks{Institute of Mathematical Finance, Ulm University, Helmholtzstrasse 18, 89075 Ulm, Germany; \url{robert.stelzer@uni-ulm.de}} \ and Johanna Vestweber\thanks{Institute of Mathematical Finance, Ulm University, Helmholtzstrasse 18, 89075 Ulm, Germany; \url{johanna2912@googlemail.com}}}

\title{Geometric Ergodicity of Affine Processes on Cones.}
\maketitle

\begin{abstract}

For affine processes on finite-dimensional cones, we give criteria for geometric ergodicity - that is exponentially fast convergence to a unique stationary distribution. Ergodic results include both the existence of exponential moments of the limiting distribution, where we exploit the crucial affine property, and finite moments, where we invoke the polynomial property of affine semigroups. Furthermore, we elaborate sufficient conditions for aperiodicity and irreducibility. Our results are applicable to Wishart processes with jumps on the positive semidefinite matrices, continuous-time branching processes with immigration in high dimensions, and classical term-structure models for credit and interest rate risk.

\end{abstract}

\begin{tabbing}
\end{tabbing}

\vspace{0.3cm}


\section{Introduction}\label{sec: intro}

Affine processes are very popular, due to their flexibility and high analytical tractability: While they enjoy all classical features of general Markov processes in continuous time (e.g., state-dependent diffusive and jump behaviour), the Backward Kolmogorov PDEs simplify to non-linear ODEs for exponential initial data. On the canonical state spaces $\mathbb R_+^m\times\mathbb R^n$ they have been fully characterized by Duffie, Filipovi\'c and Schachermayer, in the seminal paper \cite{duffie2003}. Particular affine models are well-known from financial applications (see e.g. \cite{bns2001,duffieetal2000,dai2000,heston1993closed,CIR1985}). They have been generalized to the time inhomogeneous case \cite{filipovic2005} and to other state spaces \cite{cuchiero2016affine,cuchieroteichmann2013,kellerresseletal2013,cuchiero,SpreijVeerman2012}. Of particular interest as state space are cones, especially the cone of positive semi-definite matrices (cf.~\cite{cuchiero,Alfonsi2015,Mayerhofer2009,mayerhofer2012b}). Such processes are used to model financial markets, where several assets
exhibit stochastic (co)volatilities. The typical models derive stochastic volatility from the positive-definite Wishart process (see e.g. \cite{bru,mayerhofer2012,donati2008,donatietal2004,graczyk2006,AlfonsiKebaierRey2017,BaldeauxPlaten2013,daFonsecaGrasselli2011,GourierouxSufana2010} and references therein), the Wishart model enriched with jumps (\cite{Mahoney2016,LeippoldTrojani2008,Brangeretal2018}) and the Ornstein-Uhlenbeck type stochastic volatility model (see e.g. \cite{MuhleKarbePfaffelStelzer2009,barndorffstelzer2007,BenthVos2013,GranelliVeraart2016}). Wishart and general positive semi-definite affine models have also been successfully used to model interest rates (e.g. \cite{Biagini2018,GnoattoGrasselli2014,GrasselliMiglietta2016,ChiarellaFonsecaGrasselli2014}). 
For stochastic processes the existence and uniqueness of stationary solutions as well as convergence to the stationary solution is of high interest. Geometric ergodicity means that convergence to a unique stationary distribution happens in total variation exponentially fast. It ensures fast convergence to the stationary regime in simulations and paves the way for statistical inference. By the same argument as in \cite[Proof of Theorem 4.3, Step 2]{masuda2004} geometric ergodicity and the existence of some $p$-moments of the stationary distribution guarantee exponential $\beta$-mixing (and thus strong mixing) for Markov processes. This, in turn, implies functional central limit theorems for the process (see for instance \cite{doukhan}), and thus yields asymptotic normality of estimators.\footnote{A typical statistical  approach  selects a parametric (sub)model and uses an estimation approach like quasi-maximum-likelihood or generalized methods of moments (see \cite{hansen82,Matyas1999,Hall2005}). The parametrization  needs to be sufficiently smooth and  identifiable. If one assumes that the observations are following the stationary model as well as that the true parameter is included in the parameter space and  if one chooses the parametric set-up additionally such that for every parameter one has sufficiently fast strong mixing, this implies that the quasi-likelihood or the empirical moments converge to the corresponding expected values when the number of observations/the time horizon goes to infinity and that central limit theorems hold provided sufficiently many moments are finite. Combining this with the identifiability and regularity of the parametrization then typically implies consistency and asymptotic normality of the estimators. For examples of such approaches see \cite{BrockwellSchlemm2011,Schlemm:Stelzer:2010b,genoncatalotjeantheaularedo00,Haugetal2007}. Strong mixing allows further to obtain an asymptotic theory for estimators of the spectral density, see \cite{rosenblatt1984}.  Other strands of the literature like the series of papers \cite{GaltchoukPergamenshchikov2014,GaltchoukPergamenshcikov2015,GaltchoukPergamenshchikov2007,GaltchoukPergamenshchikov20011} consider the non-parametric estimation of a drift  in a one-dimensional diffusion model. Their approach needs geometric ergodicity in a uniform manner over the parameter space. We conjecture that in order to use Theorem 2.1 of \cite{GaltchoukPergamenshchikov2014} one could strengthen our arguments to obtain uniform Lyapunov drift conditions over suitable ``compact'' parameter spaces. However, establishing the additionally needed ``uniform smallness (minorization)'' condition seems to be  intricate. Thus uniformity of geometric ergodicity results over parameter spaces are beyond the scope of the present paper and left for future research.}

In this paper we investigate geometric ergodicity of affine processes on proper, closed convex cones $K$ in finite dimensional vector spaces, by using the stability concepts for Markov processes (cf.~\cite{MT1,MT2,MT3,DMT1995}). In particular, we construct natural exponential and polynomial test functions to establish a Foster-Lyapunov drift condition, and we are able to characterize those affine processes
which satisfy the latter condition in terms of a sufficiently fast decay of the linear drift term. To this end, we develop a new result for linear maps $A$ with positive resolvent, i.e. $(\lambda-A)^{-1}(K)\subseteq K$,
extending a result by Ky Fan, \cite[Theorem 5']{fan1958topological} for the natural cones $\mathbb R_+^n$. Thus, for aperiodic and irreducible affine processes, we
find easy-to-check conditions for geometric ergodicity. For affine jump-diffusions on symmetric cones whose jumps are of compound-Poisson type with state-dependent jump intensity, we provide sufficient conditions for irreducibility and aperiodicity by relating them to diffusion processes killed at an exponential rate. Likewise, we establish easy to check sufficient conditions for irreducibility and aperiodicity in the finite activity pure jump case. Finally, we apply our general results to the special case of Wishart diffusions (with jumps) as well as the standard cone $\mathbb{R}_+^m$ and consider a simple case for the general state space $K\times \mathbb{R}^n$, for which no general theory on affine processes yet exists.

The literature on geometric ergodicity for affine processes on cones is sparse and either limited to the canonical state space,
or to Ornstein-Uhlenbeck type processes:
\begin{itemize}
\item For $K=\mathbb R_+$ \cite{jin2016positive,JinKremerRuediger2017} prove Harris recurrence and exponential ergodicity for the basic affine jump-diffusion, which arises as a default intensity model in credit risk \cite{duffie2001risk}, and slight extensions. These processes may not have state-dependent jumps. The proofs of \cite{jin2016positive} are technical and use
the explict form of the transition density. The result is a special case of our theory (see Remark \ref{rem: app dim 1}).
\item On the state space $\mathbb R_+\times\mathbb R$ \cite{barczy2014stationarity} studies geometric ergodicity of a mean-reverting two-factor affine jump--diffusion process on a half-space. We give a generalization
to affine diffusion models on $K\times\mathbb R^n$, where $K$ is a proper closed convex cone (See Theorem \ref{megageil} and Remark \ref{megasuper} \ref{nacka bazzi}).
\item Ornstein-Uhlenbeck processes driven by L\'evy noise belong to the affine class. They are cone-valued when driven by appropriate cone-valued L\'evy processes and the drift ensures that the cone cannot be left. For these processes existence of and convergence to stationary distributions is well-understood (see \cite{satoyamazato1984} and \cite{masuda2004} for geometric ergodicity).
\item In the Wishart case only criteria for convergence in distribution to the stationary case and ergodicity, in the sense that laws of large numbers hold, are known (see \cite{AlfonsiKebaierRey2017}).
\item While finalizing the present paper, \cite{ZhangGlynn2018} became available. This paper considers geometric ergodicity on the canonical state space $\mathbb{R}_+^m\times \mathbb{R}^n$. It requires finite activity jumps, strong solutions to the SDEs associated to affine processes and diagonal diffusion terms ensuring a kind of uniform ellipticity on the whole state space. If there are no state dependent jumps, it establishes positive Harris recurrence under logarithmic moment conditions and geometric ergodicity under $p$-th moment conditions with $p>0$. In the case of state dependent jumps it requires at least first moment conditions.
\end{itemize}

The present paper is primarily dedicated to affine processes on high-dimensional cones, and we allow for general, possibly infinite variation, jump behaviour and state-dependent compensators. We do not need the affine processes to be strong solutions to SDEs. Moreover, we also obtain results on exponential moments which are particularly relevant and natural in an affine context. To ensure irreducibility and aperiodicity we use appropriate controllability conditions which are weaker than uniform-ellipticity like conditions.
\subsection*{Program of the paper}
The preliminary Section \ref{sec: prep} is divided into three subsections. In Section \ref{sec: order} we develop a characterization of linear maps whose eigenvalues have strictly negative
parts (Theorem \ref{thm: linag}). This result plays a crucial part in the construction of suitable test functions for which the generator of affine processes satisfies the Foster-Lyapunov drift condition (Section \ref{sec:affine_ge}). In Section \ref{sec: affine} we recall the definition of affine processes on cones, as provided by \cite{cuchiero2016affine},
and summarize some properties. Section \ref{sec: ergodicity} recalls the definition of Harris recurrence and geometric ergodicity, and sufficient criteria for the latter property.

The main ergodicity results are developed in Section \ref{sec:affine_ge}. We split this section into two parts, for exponentially affine test functions (Section \ref{sec: exp}, Theorem \ref{prop:FLdrift_expo}) and polynomial ones (Section \ref{sec: pol}, Theorem \ref{prop:fldrift_trace}). 

Section \ref{sec: irr} provides criteria for aperiodicity and irreducibility for affine jump-diffusions with hypoelliptic main symbol and pure jump processes.

The final Section \ref{sec: appl} is dedicated to apply the theory of Sections \ref{sec:affine_ge} and \ref{sec: irr} to particularly relevant special cases: Wishart processes (with jumps) on the cone of positive semi-definite $d\times d$ matrices, affine processes on the natural cone $\mathbb R_+^n$,
and affine term-structure models on $K\times \mathbb R^n$, where $K$ is a proper closed convex cone. By attaching a linear space to the cone $K$, the state space is not proper anymore and thus falls outside
the cones considered before. 
\section{Preparatory statements}\label{sec: prep}
\subsection{Order preserving maps}\label{sec: order}
Throughout the paper, we use the notation $\mathbb R_+$ for the non-negative real line, and $\mathbb C$ for the complex numbers.

Let $V$ be a finite dimensional linear space $V$ with inner product $\langle \,,\,\rangle$. We identify, by virtue of Riesz' representation theorem, each element $\varphi$ of the dual $V^*$
by the unique $u\in V$ such that $\varphi(x)=\langle u,x\rangle$ for any $x\in V$. Hence we shall not distinguish between $V$ and $V^*$ in the following.

Let $K\subset V$ be a closed convex cone. We denote by $\preceq$ the induced partial order, that is for $x,y\in V$, $x\preceq y$ if and only if $y-x\in K$. 
$K$ is called proper, if $K\cap \{-K\}=\{0\}$, and generating, if $K-K=V$. The dual cone of $K$ is defined by $K^*:=\{u\in V\mid \langle u,x\rangle\geq 0\text{ for all } x\in K\}$. 
\begin{lemma}\label{lem: nonempty}
Let $K$ be a generating proper closed convex cone. Then
\begin{enumerate}
\item \label{cone prop 1} $K^*$ has non-empty interior $(K^*)^\circ$, and
\begin{equation}\label{eq: char dual}
(K^*)^\circ=\{u\in V\mid \langle u,x\rangle>0,\quad\forall x\in K\setminus\{0\}\}.
\end{equation}
\item \label{cone prop 2} $K^*$ is a generating proper closed convex cone. 
\item \label{cone prop 3} $K$ has non-empty interior and
\[
K^\circ=\{x\in V\mid \langle u,x\rangle>0,\quad\forall u\in K^*\setminus\{0\}\}.
\]
\end{enumerate}
\end{lemma}
\begin{proof}
The first property holds in view of \cite[Proposition I.1.4]{faraut1994analysis}.

$K^*$ is generating, because $K$ is proper (\cite[Corollary I.1.3]{faraut1994analysis}. Since $(K^*)^*=\overline K$, and $K$ is a closed generating cone, it follows that $K^*$ is proper (see \cite[Corollary I.1.3]{faraut1994analysis}). Hence \ref{cone prop 2} is proved.

Property \ref{cone prop 3} is an application of Property \ref{cone prop 1} applied to the dual cone, taking into account that $(K^*)^*=\overline K=K$ and using Property \ref{cone prop 2}.
\end{proof}
In dimensions higher than one, comparison arguments for solutions of ODEs only hold, when the vector fields have a special property, namely quasimonotonicity (besides enough regularity to allow uniqueness). We recall this property next:
\begin{definition}
A map $A: V\supseteq U\rightarrow V$ is \emph{quasimonotone increasing (qmi)} with respect to $K$, if 
\begin{equation}\label{eq:def qmi}
x,y\in U,\quad y-x\in K,\quad v\in K^*,\quad \langle v,y-x\rangle=0\quad\Rightarrow\quad \langle v,A(y)-A(x)\rangle\geq 0. 
\end{equation}
\end{definition}
\begin{remark}\label{rem both qmi}
\begin{enumerate}
\item Note that for linear maps $A: V\rightarrow V$, this property simplifies to
\[
x\in K,\quad v\in K^*,\quad \langle v,x\rangle=0\quad\Rightarrow\quad \langle v, Ax\rangle\geq 0.
\]
\item Clearly, a linear map $A:V\rightarrow V$ is 
qmi with respect to $K$ if and only if $A^\top$ is qmi with respect to $K^*$, where $A^\top$ is the adjoint of $A$, defined by
\[
\langle x,A^\top v\rangle=\langle Ax,v\rangle,\quad x,y \in V.
\]
\end{enumerate}
\end{remark}
For a linear map $A: V\rightarrow V$, let $\sigma(A)$ denote its spectrum. We use the notation $\tau(A)=\max\{\Re(\lambda),\quad \lambda\in\sigma(A)\}$ for the spectral bound, $\rho(A):=\max\{\vert \lambda\vert,\quad \lambda\in\sigma(A)\}$ for the spectral radius and $e^{tA}$ for the matrix exponential. Clearly $\tau(A)=\tau(A^\top)$, because $A$ and $A^\top$ have the same eigenvalues. Furthermore, $\lambda>\tau(A)$ implies
that $\lambda\not\in \sigma(A)$, whence $\lambda-A$ is an isomorphism. The following is a fundamental characterization of
qmi maps in terms of their resolvent:
\begin{proposition}\label{prop 1}
Let $A: V\rightarrow V$ be a linear map, and $K$ be a generating, proper closed convex cone. The following are equivalent:
\begin{enumerate}
\item \label{qmichar1} $A$ is qmi with respect to $K$.
\item $e^{tA}(K)\subseteq K$.
\item For any $\lambda>\tau(A)$, $(\lambda-A)^{-1}(K)\subseteq K$.
\item \label{qmichar4} For any $\lambda>\tau(A)$, $(\lambda-A^\top)^{-1}(K^*)\subseteq K^*$.
\end{enumerate}
\end{proposition}
\begin{proof}
The equivalence of the first three conditions follows from \cite[Satz 1, (2), (3) and (5), $\alpha=1$]{elsner1974quasimonotonie}. Note that \cite[(3)]{elsner1974quasimonotonie} is our definition of qmi, whereas \cite{elsner1974quasimonotonie} uses an equivalent one. Furthermore, the statement was shown for matrices, but we use the obvious modification for linear maps. 

By Remark \ref{rem both qmi}, $A$ is qmi with respect to $K$, if and only if $A^\top$ is qmi with respect to $K^*$. Therefore, statement \ref{qmichar4} is equivalent to \ref{qmichar1}, and can be proved similarly, by replacing the role of $K$ and $K^*$, realizing that due to Lemma \ref{lem: nonempty} the dual cone $K^*$ has non-empty interior and thus $K^*$ satisfies the standing assumption (e) of \cite{elsner1974quasimonotonie}. 
\end{proof}
Another ingredient in our main statement of this section (Theorem \ref{thm: linag}) is the celebrated
theorem by Krein-Rutman \cite{krein1948linear}:
\begin{theorem}\label{th: Krein Rutman}
Let $X$ be a Banach space and $C\subset X$ be a convex cone such that $C-C$ is dense in $X$. Let $T: X\rightarrow X$ be a compact operator which is positive, meaning $T(C)\subseteq C$, and assume that its spectral radius $\rho(T)$ is strictly positive. Then there exists $x\in C\setminus \{0\}$ such that $Tx=\rho(T)x$.
\end{theorem}
We require an elementary lemma:
\begin{lemma}\label{lem eval qmi}
Let $A:V\rightarrow V$ be qmi with respect to a generating, proper closed convex cone. Then $\tau(A)=\max\{\Re(\lambda)\mid \lambda\in \sigma(A)\}$ is an eigenvalue of $A$.
\end{lemma}
\begin{proof}
Assume, for a contradiction, $\tau(A)\not\in\sigma(A)$. Then there exist $r\in\mathbb{N}, \mu_1,\dots,\mu_r\neq 0$ such that 
\[
\sigma_A^T:=\sigma(A)\cap \{\tau(A)+i\mathbb R\}=\{\tau(A)+i\mu_i,\quad 1\leq i\leq r\}.
\]
Hence there exists $t_0>0$ such that $t_0 \mu_i\not \in \mathbb Z$ for any $1\leq i\leq r$. By the Jordan normal form of $A$,
all eigenvalues of $P:=e^{t_0 A}$ are of the form $e^{t_0\lambda}$, where $\lambda\in \sigma(A)$. Therefore, the spectral radius of $P$
equals $\rho_P:=e^{t_0 \tau(A)}$, but since $e^{t_0(\tau(A)+i\mu_i)}\neq e^{t_0\tau(A)}=\rho_P$ for any $1\leq i\leq r$, the spectral radius
is not an eigenvalue of $P$. On the other hand, since $A$ is qmi, $P$ is a positive map, hence by Theorem \ref{th: Krein Rutman} $\rho(P)$
must be an eigenvalue, a contradiction. We conclude that indeed $\tau(A)\in\sigma(A)$.
\end{proof}
The main statement of this section is the following extension of \cite[Theorem 5', (a), (c) for $K=\mathbb R_+^n$]{fan1958topological} to general cones $K$:
\begin{theorem}\label{thm: linag}
Let $K$ be a generating, proper closed convex cone, and $A: V\rightarrow V$ be a qmi linear map with respect to $K$. The following are equivalent:
\begin{enumerate}
\item \label{qmi 1} $\tau(A)<0$.
\item \label{qmi 2} There exists $v\in (K^*)^\circ$ such that $A^\top v\in -(K^*)^{\circ}$.
\end{enumerate}
\end{theorem}
\begin{proof}
\ref{qmi 1} $\Rightarrow$ \ref{qmi 2}:\\
$A^\top$ is qmi with respect to $K^*$, and clearly $\tau(A)=\tau(A^\top)$. Pick $\tau(A)<\lambda<0$. By Proposition \ref{prop 1}, $(\lambda-A^\top)^{-1} (K^*)\subseteq K^*$, hence by Lemma \ref{lem: nonempty}, there exists $v\in (K^*)^{\circ}$ such that $\eta^*:=(\lambda-A^\top)^{-1}v\in K^*$, and clearly $v\neq 0$. Therefore,
$\lambda \eta^*-A^\top\eta^*=(\lambda-A^\top) \eta^*=v$. We conclude that $-A^\top\eta^*=v-\lambda\eta^*\in K^*+(K^*)^{\circ}=(K^*)^{\circ}$.

\ref{qmi 2} $\Rightarrow$ \ref{qmi 1}:\\ Pick a $\lambda>\tau(A)$. By Proposition \ref{prop 1}, the operator
$P:=(\lambda-A)^{-1}$ is positive on $K$, and since $0$ is not an eigenvalue of a linear isomorphism, the spectral radius of $P$ is strictly larger than $0$. Hence by Theorem \ref{th: Krein Rutman}, there exists 
$x\in K\setminus\{0\}$ such that $(\lambda-A)^{-1}x=\rho_\lambda x$, and $\rho_\lambda:=\max\{\vert \mu\vert\mid \mu \in \sigma(\lambda-A)^{-1})\}$. A straightforward check reveals that $\tau:=\lambda-\frac{1}{\rho_\lambda}\in\sigma (A)$ and $A x=\tau x$. Furthermore,
\[
\tau=\lambda-\frac{1}{\rho_\lambda}=\max\{\lambda-\vert \lambda-\mu\vert\mid\mu\in\sigma(A)\}.
\]
The two functions $\sigma(A)\rightarrow \mathbb R$, 
\[
\mu\mapsto \lambda-\vert\lambda-\mu\vert,\quad\mu\mapsto \Re(\mu),
\]
agree on the real line, and the first one has its maximum at the real value $\tau\in \sigma(A)$, while the second one has its maximum at $\tau(A)\in\sigma(A)$ (see Lemma \ref{lem eval qmi}). Therefore the maxima of the two functions agree, i.e.,
\begin{align*}
\tau&=\lambda-\frac{1}{\rho_\lambda}=\max\{\lambda-\vert \lambda-\mu\vert\mid\mu\in\sigma(A)\}\\
&=\max\{\lambda-\vert \lambda-\mu\vert\mid\mu\in\sigma(A)\cap \mathbb R\}=\max\{\Re(\mu)\mid\mu\in\sigma(A)\cap \mathbb R\}
\\&=\max\{\Re(\mu)\mid\mu\in\sigma(A)\}=\tau(A).
\end{align*}
It thus remains to show that $\tau (A)<0$: To this end, let $v\in (K^*)^\circ$ such that $A^\top v\in -(K^*)^{\circ}$. Then
\[
\langle A^\top v,x\rangle=\langle v, Ax\rangle=\langle v, \tau(A) x\rangle=\tau(A)\langle v,x\rangle.
\]
Since $\langle v,x\rangle>0$ and $\langle A^\top v,x\rangle<0$ due to Lemma \ref{lem: nonempty}, $\tau(A)$ must be strictly negative, and we are done.
\end{proof}
\subsection{Affine processes.}\label{sec: affine}
In the following we give the definition of affine processes on cones, as provided by \cite{cuchiero2016affine},
and summarize some properties. The state space $K$ will be a generating, proper closed convex cone in a finite dimensional vector space $V$, as in the previous section.\footnote{The generating property, $V=K-K$ is technically necessary to identify the generalized Riccati differential equations for the exponents $\phi$, $\psi$.} Adjoined to the state space $K$ is a point $\Delta \notin K$, the cemetery state, and we set $K_{\Delta}=K\cup\{\Delta\}$ for the one-point compactification. For a Markov process $X$ on $K$,
its transition function is denoted by
$(p_t(x,\cdot))_{t\geq 0, x \in K}$. We can extend to $K_{\Delta}$ by introducing
\[
p_t(x,\{\Delta\})=1-p_t(x,K), \quad p_t(\Delta, \{\Delta\})=1,
\]
for all $ t \in \mathbb R_+$ and $x \in K$, with the convention $f(\Delta)=0$ for any function $f$ on $K$. 
\begin{definition}\label{def:affineprocessK}
A time-homogeneous Markov process $X$ relative to some filtration
$(\mathcal{F}_t)$ with state space $K$ (augmented by $\Delta$) and
transition kernels $(p_t(x,d\xi))_{t
\geq 0,x \in K}$ is called \emph{affine} if
\begin{enumerate}
\item it is stochastically continuous, that is, $\lim_{s\to t}
p_s(x,\cdot)=p_t(x,\cdot)$ weakly on $K$ for every $t \geq 0$ and $x\in K$, and
\item\label{def:affineprocess2K} its Laplace transform has exponential-affine
dependence on the initial state. This means that there exist functions
$\phi:\mathbb R_+ \times K^{\ast} \to \mathbb R_+$ and $\psi:\mathbb R_+ \times K^{\ast} \to V $
such that
\begin{align}\label{eq:affineprocessK}
\int_{K}e^{-\langle u, \xi
\rangle}p_t(x,d\xi)=e^{-\phi(t,u)-\langle \psi(t,u),x\rangle},
\end{align}
for all $x\in K$ and $(t,u) \in \mathbb R_+ \times K^{\ast}$.
\end{enumerate}
\end{definition}
The following statement recalls the basic properties of affine processes on cones, in particular
the description of $(\phi,\psi)$ as solutions for so-called generalized Riccati differential equations (cf. \cite[Theorem 2.4, Part I]{cuchiero2016affine})
\begin{theorem}\label{th: main theorem}
Let $X$ be an affine process on $K$. Then $X$ is a Feller process, the functions $\phi$ and $\psi$ given in
\eqref{eq:affineprocessK} are differentiable with respect to time
and satisfy the generalized Riccati equations for $u \in K^{\ast}$,
that is,
\begin{subequations}\label{eq:Riccati}
\begin{align}
\frac{\partial \phi(t,u)}{\partial t}&=F(\psi(t,u)), &\quad &\phi(0,u)=0,\label{eq:RiccatiF}\\
\frac{\partial \psi(t,u)}{\partial t}&=R(\psi(t,u)), &\quad &\psi(0,u)=u \in K^{\ast}, \label{eq:RiccatiR}
\end{align}
\end{subequations}
where $F(u)=\partial_t\phi(t,u)|_{t=0}$ and $R(u)=\partial_t \psi(t,u)|_{t=0}$.
Moreover, relative to some truncation function\footnote{A truncation function is continuous, bounded in norm by $1$ and equal to the identity in a neighborhood of the origin.} $\chi$, there exists a parameter set $(Q,b,B,c,\gamma,m,\mu)$
such that the functions $F$ and $R$ are of the form
\begin{subequations}\label{eq:FRLevyK}
\begin{align}
F(u)&=\langle b,u \rangle + c -\int_{K} \left(e^{-\langle u, \xi\rangle}-1\right) m(d\xi),\label{eq:FLevyK}\\
R(u)&=-\frac{1}{2}Q(u,u)+B^{\top}(u)+\gamma-\int_{K} \left(e^{-\langle u, \xi\rangle}-1+ \langle \chi(\xi),u \rangle \right) \mu(d\xi),\label{eq:RLevyK}
\end{align}
\end{subequations}
where
\begin{enumerate}
\item\label{eq:const_drift1} $b \in K$,
\item\label{eq:const_killing1} $c \in \mathbb R_+$,
\item\label{eq:const_measure1} $m$ is a
Borel measure on $K$ satisfying $m(\{0\}) = 0$ and
\[
\int_K \left(\|\xi\| \wedge 1\right)m(d\xi)< \infty,
\]
\item\label{eq:lin_diffusion1} $Q: V \times V \rightarrow V$ is a symmetric bilinear function
such that for all $v \in V$, $Q(v,v) \in K^{\ast}$ and $\langle x, Q(u,v) \rangle=0$, whenever $\langle u, x\rangle=0$ for $u \in K^{\ast}$ and $x \in K$,
\item\label{eq:lin_killing1} $\gamma \in K^{\ast}$,
\item\label{eq:lin_measure1} $\mu$ is a $K^{\ast}$-valued $\sigma$-finite Borel measure on $K$ satisfying $\mu(\{0\}) = 0$,
$
\int_K\left(\|\xi\|^2 \wedge 1 \right)\langle x,\mu(d\xi)\rangle < \infty
$ for all $x \in K$,
and
\[
\int_K\langle \chi(\xi), u \rangle \langle x, \mu(d\xi)\rangle < \infty \textrm{ for all $u \in K^{\ast}$ and $x \in K$ with $\langle u, x\rangle=0$},
\]
\item\label{eq:lin_drift1} $B^{\top}: V \rightarrow V$ is a linear map, satisfying
\[
\langle x, B^{\top}(u)\rangle - \int_K \langle \chi(\xi), u \rangle \langle x, \mu(d\xi)\rangle \geq 0 \textrm{ for all $u \in K^{\ast}$ and $x \in K$ with $\langle u, x\rangle=0$}.
\]
\end{enumerate}
\end{theorem}
By the defining property \eqref{eq:affineprocessK} and the generalized Riccati differential equations
\eqref{eq:Riccati}, for any $x\in K$, and $u\in K^*$ we have for $f_u(x):=e^{-\langle u,x\rangle}
$ that
\begin{equation}\label{eq: gen bounded}
\lim_{t\downarrow 0}\frac{P_t f_u(x)-f_u(x)}{t}=(-F(u)-\langle R(u),x\rangle)f_u(x).
\end{equation}
This limit also holds in supremum norm, due to the Feller property of $X$ (we shall denote by $\mathcal D(\mathcal A)$ the domain of the infinitesimal generator $\mathcal{A}$ of $X$, that is the generator of the
associated Feller semigroup).
\begin{proposition}\label{prop action exp}
For any $u\in (K^*)^\circ$, we have $ f_u\in \mathcal D(\mathcal A)$, and
\begin{equation}
\mathcal A f_u(x)=(-F(u)-\langle R(u),x\rangle)f_u(x).
\end{equation}
\end{proposition}
\begin{proof}
The statement is a slight adaption of the proof of \cite[Proposition 4.12]{cuchiero} to cones $K$, the only 
difference being that our cones are not necessarily symmetric\footnote{A convex cone is called symmetric, if it is self-dual and if the linear automorphism group acts transitively on it.}: By compactness of the unit sphere $\{x\in K\mid\|x\|=1\}$, for any $u\in (K^*)^\circ\subset K^*$, there exists a constant $\lambda_u>0$ such that $\langle u,x\rangle\geq \lambda_u \|x\|$ for any $x\in K$. Hence $f_u\in C_0(K)$. The rest of the proof is exactly the same as \cite[Proposition 4.12]{cuchiero} .
\end{proof}
We end this section by giving a characterization of conservative affine processes, and an easy-to-check
sufficient condition for non-explosion.\footnote{A process is non-explosive (or conservative), if the transition function does not charge the cemetery state $\Delta$.}
\begin{proposition}\label{prop cons}
Let $X$ be an affine process on $K$. The following are equivalent:
\begin{enumerate}
\item \label{cons 1} $X$ is conservative.
\item \label{cons 2} $c=0$, $\gamma=0$, and $0$ is the only $K^*$--valued solution of the generalized Riccati differential equation
\begin{equation}\label{eq: start zero}
\partial_t\psi(t)=R(\psi(t)),\quad \psi(0)=0.
\end{equation}
\end{enumerate}
In particular, if $c=0$, $\gamma=0$ and for some $x_0\in K^\circ$,
\begin{equation}\label{cons 3}
\int_{\|\xi\|\geq 1} \|\xi\| \langle \mu(d\xi),x_0\rangle<\infty,
\end{equation}
then \ref{cons 2} is satisfied, and thus $X$ is non-explosive.
\end{proposition}
\begin{proof}
The characterization \ref{cons 1} $\Leftrightarrow$ \ref{cons 2} can be proved along the lines \cite[Theorem 3.4]{mayerhofermuhle}, using the quasi-monotonicity of the function $u\rightarrow R(u)$ with respect to $K^*$ (see \cite[Proposition 3.12]{cuchiero2016affine}).

Note that the positive measure $(\|\xi\|^2 \wedge 1) \langle\mu(d\xi),x_0\rangle$ is finite, whence by \cite[Lemma A.1 for $k=1$]{duffie2003} condition \eqref{cons 3} implies that 
\[
\varphi_x: V\rightarrow \mathbb C:\quad v\mapsto\langle R(iv),x\rangle
\]
is $C^1$ for $x=x_0$, and thus in particular, locally Lipschitz. Since $x_0\in K^\circ$, for any $x\in K$ there exists $\lambda_x>0$ such that $x\preceq \lambda_x x_0$, whence \eqref{cons 3} indeed holds for any $x\in K$, and therefore $\varphi_x$ is locally Lipschitz, for any $x\in K$. Since $K-K=V$ we infer that $V\rightarrow\mathbb C$, $v\mapsto R(iv)$ is locally Lipschitz. Suppose, $\psi(t)$ is a $K^*$--valued solution of \eqref{eq: start zero}. Then $\chi(t)=-i\psi(t)$ solves the complex-valued generalized
Riccati differential equation
\[
\partial_t \chi=R(i\chi),\quad \chi(0)=0.
\]
By local Lipschitz continuity of $R$ on $iV$, $\chi\equiv 0$, thus $\psi\equiv 0$. Therefore, \ref{cons 2}
is satisfied and thus $X$ is conservative.\footnote{Our argument is longer than similar proofs (e.g., \cite{cuchiero2016affine, mayerhofermuhle, cuchiero}) which
overlook the fact that Lipschitz continuity does not necessarily hold in the real domain near zero, but only for the purely imaginary argument. In fact, for $R(u)$ being defined in a real neighborhood of zero,
one needs the existence of exponential moments.}
\end{proof}
\subsection{Markov processes and ergodicity}\label{sec: ergodicity}
In this section, $X=(X_t)_{t\ge0}$ is a continuous-time Markov process (please see the standard literature like \cite{DMT1995,MT2,MT3,MT} for definitions of notions for Markov processes not explained here) on a locally compact, separable metric space $D$ with transition probabilities $p_t(x,A)=\mathbb P^x[X_t \in A]$ for $x\in D, A \in \mathcal B(D)$. We assume that $X$ is a non-explosive Borel right process. 
We denote with $\pi$ the unique invariant measure of $X$, if it exists. 
Exponential ergodicity means that the total variation distance of the transition semigroup and the invariant measure converge to zero independent of the initial state $x$ and with an exponential rate. A ``seemingly stronger"
concept\footnote{The phrase ``seemingly stronger" stems from \cite{DMT1995}. In fact, only for $f=1$ , the two formulations coincide -- then the total variation norm of a signed measure equals its $f$-norm. However, due to \cite[Chapter 3 and Theorem 2.1]{DMT1995}, for Markov chains, $f$-geometric ergodicity is equivalent with
weaker types of geometric ergodicity.} of this is the $f$-uniform ergodicity, which demands the convergence in the $f$-norm\footnote{The $f$-norm (of a signed measure $\mu$) is defined for every measurable function from the state space $D$ to $[1, \infty)$ by $\|\mu\|_f:=\sup_{|g|\le f} | \int \mu(dy)g(y)|$. }
instead of the total variation norm: 
\begin{definition}[$f$-uniform ergodicity, {\cite[Chapter 3]{DMT1995}}]
$X$ is called \emph{$ f$-uniformly ergodic}, if a measurable 
function $f: D \to [1,\infty)$ exists such that for all $x \in D$ 
\begin{equation}
\| p_t(x,.) - \pi \|_f \le f(x) C \rho^t, \ t\ge0
\end{equation}
holds for some $C<\infty, \rho <1$.
\end{definition}
Another concept is a probabilistic form of stability, the so called Harris recurrence. 
\begin{definition}[Harris recurrence, {\cite[Chapter 2.2]{MT2}}]
\begin{enumerate}Let $\eta_A := \int_0^\infty 1_{\{X_t \in B\} } dt$ be the occupation time and $\tau_A:= \inf \{t\ge0~:~ X_t \in A \}$ be the first hitting time of $A$. 
\item $X$ is called \emph{Harris recurrent}, if either
\begin{itemize}
\item $\mathbb P^x[\eta_A = \infty]=1$ whenever $\phi(A)>0$ for some $\sigma$-finite measure $\phi$, or
\item $\mathbb P^x[\tau_A < \infty]=1$ whenever $\mu(A)>0$ for some $\sigma$-finite measure $\mu$. 
\end{itemize}
\item Suppose that $X$ is Harris recurrent with finite invariant measure $\pi$, then $X$ is called \emph{positive Harris recurrent}.
\end{enumerate}
\end{definition}
\begin{definition}[{\cite{DMT1995}}, Chapter 3]
 For a $\sigma$-finite measure $\mu$ on $\mathcal B(D)$ we call the process 
$X$ \emph{$ \mu$-irreducible} if for any $A \in  \mathcal B(X)$ 
with $\mu(A)>0$ 
\begin{equation}
 \mathbb E^x (\eta_A) > 0, \forall x \in X.
\end{equation}
\end{definition}
This is obviously the same as requiring 
$
 \int_0^\infty \mathbb{P}^x(X_t\in B) dt >0,\  \forall x \in X.
$
If $X$ is $\mu$-irreducible, there exists a maximal irreducibility measure $\psi$ 
such that every other irreducibility measure $\nu$ is absolutely continuous with respect to $\psi$ (see \cite[p. 493]{MT2}). 
We write ${\mathcal B^+(D)}$ for the collection of all measurable subsets $A\in \mathcal B(X)$ with $\psi(A)>0$.
\begin{definition}[{\cite{DMT1995}}, Chapter 3] \label{def:aperiodic}
\begin{enumerate}
\item A nonempty set $C\in \mathcal B(D)$ is called \emph{small}, if there exists an $m>0$ and a nontrivial measure $\nu_m$ on $\mathcal B(D)$ such that for all $x\in C,A\in \mathcal B(D)$
 \begin{equation}
 \mathbb P^x [X_t\in A] \ge \nu_m(A).
 \end{equation} 
 \item A $\psi$-irreducible Markov process is called \emph{aperiodic} if for some small set $C \in \mathcal B^+(X)$ there exists a $T$ 
 such that 
 $\mathbb P^x[X_t\in C] >0$ for all $t\ge T$ and all $x\in C$. 
\end{enumerate}
\end{definition}
For irreducible and aperiodic Markov processes  a sufficient condition for $f$-uniform ergodicity is the so-called Foster-Lyapunov drift condition:
\begin{theorem}[{\cite{DMT1995}}, Theorem 5.2] \label{dmtergodicity} 
Let $(X_t)_{t\ge 0}$ be a $\mu$-irreducible and aperiodic Markov process. If 
there exist constants $b,c > 0$ and a petite set $C$ in $\mathcal B(D)$ as well as a measurable function $f: 
D \to [1,\infty)$ (referred to as Lyapunov function) such that \begin{equation}
\mathcal A f \le -bf + c 1_C ,
\end{equation}
where $\mathcal A$ is the extended generator, then $X$ is $f$-uniformly ergodic.
\end{theorem}
\begin{definition}[Extended generator, {\cite[Chapter 4]{DMT1995}}]
$\mathcal D (\mathcal{A})$ denotes the set of all functions $f: D\times \R_{+} \to \R$ 
for which a function $g: D \times \R_{+} \to \R$ exists, such that $\forall x 
\in D, t>0$
\begin{align}
\label{exgen1} &\mathbb E^{x}(f(X_{t},t)) = f(x,0) +\mathbb E^{x}\left( \int_{0}^{t}g(X_{s},s)ds\right), \\
\label{exgen2} &\mathbb E^{x}\left( \int_{0}^{t}|g(X_{s},s)|ds\right) < \infty
\end{align}
holds. We write {$\mathcal{A} f := g$} and call $\mathcal{A}$ the \emph{extended generator of }{ $X$}. 
$\mathcal D(\mathcal{A})$ is called the domain of $\mathcal{A}$. 
\end{definition}
The above definition of the extended generator is the usual one in the papers on geometric ergodicity of Markov processes. Note that in other strands of the literature the definition is often different. For example, many papers only demand that $f(X_{t},t)-f(x,0)- \int_{0}^{t}g(X_{s},s)ds$ is a local martingale. This is the case e.g. in \cite{cuchiero2012polynomial}. Their Theorem 2.7 and the preceding discussion, however, ensures that their extended generator for $m$-polynomial processes with $m\geq 2$ is also the extended generator in our sense and all polynomials of order at most $m$ are in its domain.

For some $\Delta>0$ we call the Markov chain $(X_{k\Delta})_{k\in\mathbb{N}}$ a ($\Delta$-)skeleton chain.

We skip the definition of a petite set since we use the following result, which ensures petiteness of compact sets (note that only
irreducibility is assumed, but not aperiodicity):
\begin{proposition}[{\cite[Theorem 6.0.1]{MT}}]\label{prop:compactpetite}
If $Y$ is a $\Psi$-irreducible Feller chain with $\text{supp}(\Psi)$ having non-empty interior, every compact set is petite.
\end{proposition}
\section{Geometric ergodicity of affine processes} \label{sec:affine_ge}
The following two sections offer two different sets of sufficient conditions for geometric ergodicity of affine processes on $K$. The main difference between the two statements lies in the assumption of finite moments in the second one, instead of exponential ones in the first one; therefore, also the conclusions concerning the moments of the limiting distributions are different. 

Due to the characterization of affine processes via the exponential-affine form of the Fourier/La\-place transform, considering exponential moments is particularly natural and relevant. 
\subsection{Exponential moments}\label{sec: exp}
\begin{theorem}\label{prop:FLdrift_expo}
Let $X$ be a an affine process on $K$ with $c=0$, $\gamma=0$. Assume that 
\begin{enumerate}
\item \label{thm11} $X$ is $\nu$-irreducible with the support of $\nu$ having non-empty interior and aperiodic,
\item \label{thm12} there exists $\eta_0\in \mathbb (K^*)^\circ$ such that
\begin{equation}\label{eq:conditionbeta2x}
B^\top \eta_0 + \int_{K\setminus\{0\}} \langle\eta_0,\xi-\chi(\xi)\rangle \mu(d\xi) \in - (K^*)^\circ
\end{equation}
and for all $t\geq 0$, and for some $x\in K^\circ$, $\mathbb E^x[e^{\langle\eta_0, X_t\rangle}]<\infty$.
\end{enumerate} 
Then $X$ is conservative, and for any $\eta\preceq \eta_0$, and all $x\in K$
\begin{align} \label{eq:ass_exponentialm}
&\int_{\{ \|\xi\|\ge 1\}} e^{ \langle\eta, \xi\rangle} (m(d\xi)+\langle\mu(d\xi),x\rangle) < \infty
\end{align} 
and $X$ is geometrically ergodic and positive Harris recurrent. Furthermore, the stationary distribution $\pi$ has finite exponential moments which can be computed
with the affine transform formula, i.e. there exists
$h_0>0$ such that for all $v \preceq h_0\eta_0$,
\begin{equation}\label{eq: limiting ATF}
\int_{K\setminus \{0\}} e^{\langle v,\xi\rangle} \pi(d\xi) =e^{-\int_0^\infty F(\psi(s,-v))ds}.
\end{equation}
\end{theorem}
\begin{proof}
Since $\eta_0$ is in the interior of the dual cone, Condition \eqref{eq:conditionbeta2x} implies that Condition \eqref{cons 3} of Proposition \ref{prop cons} is satisfied. Hence by Proposition \ref{prop cons}, $X$ is conservative. Since $\mathbb E^x[e^{\langle\eta_0, X_t\rangle}]<\infty$, the monotonicity of the exponential implies $\mathbb E^x[e^{\langle\eta, X_t\rangle}]<\infty$ for any $\eta\preceq \eta_0$. We conclude by \cite[Theorem 2.14 (a)]{KRM2015} that \eqref{eq:ass_exponentialm} 
is satisfied for any $\eta\preceq \eta_0$ and that for any $x\in K$, any $\eta\preceq \eta_0$, and any $t\geq 0$
\begin{equation}\label{eq: ATF}
\mathbb E^x[e^{\langle \eta, X_t)\rangle}]=e^{p(t,\eta)+\langle q(t,\eta),x\rangle},
\end{equation}
where $q$ is a (not necessarily unique\footnote{In the next paragraph we avoid non-uniqueness
by scaling $\eta_0$ sufficiently.}) solution of the generalized Riccati differential equation $\partial_t q=-R(-q)$ with $q(0)=\eta$
and $\partial_t p=-F(-q)$, $p(0)=0$. In particular, $p(t,\eta)=-\int_0^t F(-q(s,\eta))ds$.\footnote{The identification $p(t,\eta)=-\phi(t,-\eta)$
and $q(t,\eta)=-\psi(t,-\eta)$ is necessary, since we use the Laplace transform for the defining affine property, whereas \cite{KRM2015} uses the characteristic function}

Now $u\mapsto R(u)$ and $u\mapsto F(u)$ are analytic functions in a neighborhood of $0$, due to \eqref{eq:ass_exponentialm} and their explicit definitions, see \eqref{eq:FRLevyK}. Pick some $x_0\in K^\circ$. It follows directly from the parametric restrictions (see Theorem \ref{th: main theorem}) that the map $u\mapsto -R(-u)$ is qmi (with respect to the dual cone $K^*$) on the domain
\[
\mathcal U:=\left\{u\in V\mid \int_{\{ \|x\|\ge 1\}} e^{ \langle u, \xi\rangle} (m(d\xi)+\langle\mu(d\xi),x_0)\rangle) < \infty\right\}.\footnote{The definition is independent of the choice of $x_0\in K^\circ$, because for any $y_0\in K^\circ$ there exists $h>0$ such that for any $u\in K^*$ we have $h\langle y_0,u\rangle\leq \langle x_0,u\rangle\leq h^{-1}\langle y_0,u\rangle$.}
\]
Furthermore, by assumption \eqref{eq:conditionbeta2x},
\[
\frac{\partial R}{\partial u}\big|_{u=0}(\eta_0)= B^\top \eta_0 + \int_{K\setminus\{0\}} \langle\eta_0,\xi-\chi(\xi)\rangle \mu(d\xi) \in - (K^*)^\circ.
\]
Thus Theorem \ref{thm: linag} is applicable and guarantees that $\tau(\frac{\partial R}{\partial u}\mid_{u=0})<0$. The latter implies
by a well known result of Lyapunov (\cite{lyapunov1992general}) that $0$ is an asymptotically stable equilbrium of the ODE
$\partial_t q=-R(-q)$. In fact, there exists $\delta>0$ such that $B_\delta(0):=\{u\in V\mid \|u\|\leq \delta\}\subset \mathcal U^\circ$ (where $R$ is analytic). Hence, for any $u\in B_\delta(0)$ the (unique) solution $q(t,u) $ exists for all times $t\geq 0$ and satisfies $\delta\geq\|q(t,u)\|\rightarrow 0$, as $t\rightarrow \infty$. Furthermore, it is geometrically stable (see e.g. \cite[Theorem 7.1]{Verhulst1990}), that is, for all $\|u\|\leq \delta$, and all $t\geq 0$, $\|q(t,u)\|\leq c e^{-dt}$ for some constants $c,d>0$. Therefore also $p(t,v)\rightarrow \int_0^\infty -F(-q(s,u))ds$
as $t\rightarrow \infty$. Therefore, for $\eta\in B_\delta(0)$, the affine transform formula \eqref{eq: ATF} holds, where $(p,q)$ is the unique, global solution of the generalized
Riccati equations. Due to asymptotic stationarity of the ODE, we have therefore shown \eqref{eq: limiting ATF} for $v\in B_\delta(0)$ (this fact is used next, while the extension to any $v\preceq h\eta_0$ is done in the last paragraph of this proof).

By Taylor's formula, we have $-R(-h\eta_0)=h DR_{\eta_0}(0)+o(h^2)$, hence there exists $0<h_0< 1$ such that for $\eta=h_0\eta_0$, $-R(-\eta)\in -(K^*)^\circ$. Pick this $h_0$ sufficiently small such that also $h_0\eta_0\in B_\delta(0)$, and set $\eta:=h_0\eta_0$. Since $\eta$ is in the interior of the dual cone, the map $H_\eta: K\rightarrow [1,\infty)$, $H_\eta(x):=e^{\langle \eta,x\rangle}$ has range equal to $[1,\infty)$. Since by assumption, $\mathbb E^x[H_{\eta_0}(X_t)]<\infty$, for any $t>0$, also $\mathbb E^x[H_{\eta}(X_t)]<\infty$ and since $\eta_0\in (K^*)^\circ$, there exists a constant $L>0$ such that for all $x\in K$
\[
\|x\| e^{\langle \eta,x\rangle}\leq L e^{\langle \eta_0,x\rangle},
\]
which implies that
\[
\mathbb E^x\left [\int_0^t  H_\eta(X_s) \vert F(-\eta)+\langle R(-\eta),X_s\rangle\vert ds\right]<\infty.
\]
Therefore, to show that $H_\eta$ belongs to the domain of the extended generator, it suffices to establish
that
\[
\mathbb E^x[H_\eta(X_t)]=H_\eta(x)+\mathbb E^x\left [\int_0^t  H_\eta(X_s) \left(-F(-\eta)+\langle -R(-\eta),X_s\rangle\right)ds\right].
\]
Taking into account the affine transform formula \eqref{eq: ATF}, it is therefore enough to prove that for any $u\in B_\delta(0)$
\begin{equation}\label{eq: super toyota}
e^{-\phi(t,u)-\langle \psi(t,u), x\rangle}-H_{-u}(x)=\mathbb E^x\left [\int_0^t  H_{-u}(X_s) \left(-F(u)-\langle R(u),X_s\rangle\right)ds\right].
\end{equation}
Note that both left and right side of \eqref{eq: super toyota} are analytic functions on $B_\delta(0)$. Since any open set is a set of uniqueness for analytic functions, the analytic functions on the left and right sides of
\eqref{eq: super toyota} must agree, if they agree only on $B_\delta(0)\cap (K^*)^\circ$ (cf.~\cite[(9.4.2)]{dieudonne1969foundations}). And indeed, the identity \eqref{eq: super toyota} holds for any $u\in (K^*)^\circ$, because in this case $H_{-u}\in \mathcal D(\mathcal A)$, which first implies\footnote{This follows directly from the Feller property as stated in Theorem \ref{th: main theorem}. For a direct proof, see \cite[Lemma 3.2]{cuchiero2016affine}.} that $H_{-\psi(s,u)}\in \mathcal D(\mathcal A)$ for any $s\geq 0$, and as the semigroup action commutes with the action of the infinitesimal generator (\cite[Proposition 1.5]{ethier2009markov}), we have
\begin{align}\label{eq: commutator}
&-F(\psi(s,u)-\langle R(\psi(s,u),x\rangle e^{\phi(s,u)}H_{\psi(s,u)}x=\mathcal A\left(\mathbb E^x[H_{-u}(X_s)]\right)=\mathbb E^x[\mathcal A H_{-u}(X_s)]\\\nonumber&\qquad=\mathbb E^x[(-F(u)-\langle R(u),X_s\rangle)H_{-u}(X_s)].
\end{align}
Using the fundamental theorem of calculus, the generalized Riccati differential equations \eqref{eq:Riccati}, identity \eqref{eq: commutator} and Fubini's theorem, we thus have for any $u\in (K^*)^\circ$,
\begin{align*}
e^{-\phi(t,u)-\langle \psi(t,u), x\rangle}-H_{-u}(x)&=\int_0^t (-\partial_s\phi(s,u)-\langle \partial_s\psi(s,u),x\rangle)e^{-\phi(s,u)-\langle \psi(s,u),x\rangle}ds\\\quad&=-\int_0^t (F(\psi(s,u)+\langle R(\psi(s,u),x\rangle)e^{-\phi(s,u)-\langle \psi(s,u),x\rangle}ds\\\qquad&=-\mathbb E^x\left[\int^t (F(u)+\langle R(u),X_s\rangle)H_{-u}(X_s)ds\right].
\end{align*}
which is \eqref{eq: super toyota}. This finishes the proof of $H_\eta$ being in the domain of the extended generator.
%
%

Since $-R(-\eta)\in-(K^*)^\circ$, there exists $\lambda_0>0$ such that $\langle R(-\eta),x\rangle\geq \lambda_0\|x\|$ (because, the continuous map $x\mapsto \langle R(-\eta),x\rangle$ attains its minimum
over the set $\{x\in K\mid \|x\|=1\}$). Hence,
\[
\mathcal A H_\eta (x)\leq H_\eta(x) (-F(-\eta)-\lambda_0 \|x\|).
\]
Select $k>0$ such that $-F(-\eta)-\lambda_0 k=-c<0$. Then for any $x\in K\setminus B$, where $B:=\{x\in K\mid \|x\|\leq k\}$
we have
\[
\mathcal A H_\eta(x)\leq -c H_\eta(x),
\]
and since $\mathcal AH_\eta$ is bounded on the compact set $B$, there exists a constant $d>0$ such that
\[
\mathcal A H_\eta(x)\leq -c H_\eta(x)+d 1_{K}(x),
\]
that is, the Foster-Lyapunov drift condition is satisfied for $H_\eta$, because $B$ is a compact set in $K$ and therefore by Proposition \ref{prop:compactpetite} petite.

We continue the proof by applying Theorem \ref{dmtergodicity}: By assumption $X$ is irreducible and aperiodic and we have shown that the Foster-Lyapunov drift condition holds. Further we know that $X$ is a Feller process, see Theorem \ref{th: main theorem}, hence it is a Borel right process, see \cite{marcusrosen}. Since $X$ is conservative by assumption it is non-explosive. Thus all assumptions of Theorem \ref{dmtergodicity} are fulfilled and thus $X$ is $H_\eta$-uniformly ergodic (whence, in particular geometrically ergodic).

Next we show that $X$ is positive Harris recurrent and that the stationary distribution has the demanded moments. For discrete time Markov chains in \cite[Theorem 3.12]{fuchs} it is shown that the Foster-Lyapunov drift condition for $(\mathcal A, H)$ implies besides geometric ergodicity also positive Harris recurrence and for the stationary distribution $\pi(H) < \infty$. Applying \cite[Theorem 5.1]{DMT1995} we get that the proven Foster-Lyapunov drift condition implies the discrete version for every skeleton chain of $X$. By the definition of positive Harris recurrence it holds also for the process $X$ and with the chosen test functions $H_\eta$ in the proof above the $\eta$-th exponential moment of the stationary distribution $\pi$ exists. 

It remains to prove \eqref{eq: limiting ATF} (we have shown above that it holds for $v\in B_\delta(0)$, and since $\eta=h_0\eta_0\in B_\delta(0)$,
it holds for $\eta$, in particular). Since $X$ is $H_\eta$-uniformly ergodic, we know that for any continuous function $f$ satisfying $\vert f\vert\leq H_\eta$, $\mathbb E^x[f(X_t)]\rightarrow \int f(\xi)\pi(d\xi)$, as $t\rightarrow \infty$, where $\pi$ is the unique stationary limit
of $X$. Since for any $v\preceq h_0\eta_0=\eta$ we have $f_v(x)=\exp(\langle v,x\rangle)\leq H_\eta(x)$, it follows that for any $v\preceq h_0\eta_0$
\[
e^{p(t,v)+\langle x,q(t,v)\rangle}=\mathbb E^x[e^{\langle v, X_t\rangle }]\rightarrow \int_K e^{\langle v,\xi\rangle}\pi(d\xi),\quad\text{as}\quad t\rightarrow\infty.
\]
Since the right hand side does not depend on the inital state $x$, we see $e^{\langle x,q(t,v)\rangle}\rightarrow 1$, i.e. $\langle x,q(t,v)\rangle\rightarrow 0 $ as $t\rightarrow \infty$,
for any $x\in K$. Since $K$ is generating ($K-K=V$), and the inner product is continuous, we obtain $\lim_{t\rightarrow \infty}q(t,v)=0$. Thus $e^{p(t,v)}=e^{\int_0^t- F(-q(s,v))ds}$ converges, as $s\rightarrow \infty$. Since the Laplace transform of the measure $\pi$ must be strictly positive, the limit is non-zero, and thus $\lim_{t\rightarrow \infty}e^{p(t,v)}=e^{\lim _{t\rightarrow \infty}p(t,v)}=e^{\int_0^\infty- F(-q(t,v))dt}$.
\end{proof}
The assumption of finite exponential moments for all times in Theorem \ref{prop:FLdrift_expo}
is implied by the existence of solutions to the generalized Riccati equations for all times, i.e. no finite blow-up
(\cite[Theorem 2.14 (b)]{KRM2015}). Therefore we can restate the Theorem as follows:
\begin{corollary}\label{cor:expmom}
Let $X$ be a an affine process on $K$ with $c=0$, $\gamma=0$. Assume that
\begin{enumerate}
\item \label{thm11cor} $X$ is $\nu$-irreducible with the support of $\nu$ having non-empty interior and aperiodic,
\item \label{thm12cor} There exists $\eta_0\in \mathbb (K^*)^\circ$ such that
\begin{equation}\label{eq:conditionbeta2xx}
B^\top \eta_0 + \int_{K\setminus\{0\}} \langle\eta_0,\xi-\chi(\xi)\rangle \mu(d\xi) \in - (K^*)^\circ
\end{equation}
and assume that 
\begin{align} \label{eq:ass_exponentialm1y}
&\int_{\{ \|x\|\ge 1\}} e^{ \langle\eta_0, \xi\rangle} (m(d\xi)+\langle\mu(d\xi),x\rangle) < \infty
\end{align} 
for some $x\in K^\circ$.
\end{enumerate} 
Then $X$ is conservative, \eqref{eq:ass_exponentialm1y} holds for any $\eta\preceq \eta_0$, and all $x\in K$ and $X$ is geometrically ergodic and positive Harris recurrent. Furthermore, the stationary distribution $\pi$ has finite exponential moments which can be computed
with the affine transform formula, i.e. there exists
$h_0>0$ such that for all $v\preceq h_0\eta_0$
\[
\int_{K\setminus \{0\}} e^{\langle v,\xi\rangle} \pi(d\xi) =e^{-\int_0^\infty F(\psi(s,-v))ds}.
\]
\end{corollary}
\begin{proof}
By monotonicity of the exponential, of course \eqref{eq:ass_exponentialm1y} also holds for any $\eta\preceq \eta_0$. Therefore it holds in a neighborhood of the origin. Hence the generalized Riccati differential equations 
are well-posed for initial data near zero. Furthermore, by \eqref{eq:conditionbeta2xx} and Theorem \ref{thm: linag}, the Frechet differential of $R(u)$ at zero has negative spectral bound $\tau$. Therefore for small enough initial data, the generalized Riccati differential equations have global solutions which tend to zero as $t\rightarrow\infty$.
According to \cite[Theorem 2.14 (b)]{KRM2015}, for any $x\in K$, $\mathbb E^x[e^{\langle\eta,X_t\rangle}]<\infty$ for initial data $\eta$ small enough. Thus the assumptions of Theorem \ref{prop:FLdrift_expo} are satisfied.
\end{proof}
\begin{remark}
\begin{enumerate}
\item It is easy to see that condition \eqref{eq:conditionbeta2x} is independent of the choice of the truncation function. The appearance of the truncation function implies that we can indeed cover cases with not only infinite activity, but even infinite variation state dependent jumps. 
\item The drift condition \eqref{eq:conditionbeta2x} involves the linear drift and the state dependent jumps. 
\item The state-independent jumps do not matter at all, \eqref{eq:conditionbeta2x} and \eqref{eq:conditionbeta2xx} only imply that they need to have finite exponential moments. In the case of an affine diffusion the Assumption (ii) of Corollary \ref{cor:expmom} boils down to the existence of an $\eta_0\in(K^*)^\circ$ such that $B^\top\eta_0\in-(K^*)^\circ$. Note that in the diffusive case the above reasoning based on \cite{KRM2015} implies that we automatically have some exponential moments.
\item For compound Poisson jumps and the cone $\R_+^m $ Condition \eqref{eq:conditionbeta2x} is identical to the drift condition of \cite{ZhangGlynn2018} which shows that it is close to being necessary as well. The latter can also be seen from the special case of Ornstein-Uhlenbeck type processes (see e.g. \cite{masuda2004,satoyamazato1984}). 
\end{enumerate}
\end{remark}
\subsection{Finite moments}\label{sec: pol}
In this section, polynomial functions are used to obtain the Foster-Lyapunov drift condition for
the extended generator of $X$. This requires to know the precise action of the generator on polynomials.
To this end we extend Proposition \ref{prop action exp} to a larger function space.

For the following paragraph we borrow some introductory material and notation from \cite[Section 4.3]{cuchiero2016affine}\footnote{While the quoted material
is stated within the framework of symmetric cones (Section 4), we only borrow the part that holds for general cones.}. 
We shall consider the tensor product $V \otimes V^{\ast}$, which we identify via the canonical isomorphism
\[
(u \otimes v)x=\langle x,v\rangle u, \quad x \in V,
\]
with the vector space of linear maps on $V$ denoted by $\mathcal{L}(V)$.
Moreover, for an element $A \in \mathcal{L}(V)$, we denote its trace by $\text{Tr}(A)$.
Observe that $\text{Tr}(A (u\otimes u))=\langle u, A u\rangle$. Indeed, by choosing an orthonormal basis $\{e_{\beta}\}$ of $V$, we have
\[
\text{Tr}(A (u\otimes u))=\sum_{\beta} \langle A^{\top}e_{\beta}, (u \otimes u)e_{\beta}\rangle
=\sum_{\beta} \langle A^{\top}e_{\beta}, \langle u, e_{\beta}\rangle u\rangle
=\sum_{\beta} \langle u, e_{\beta}\rangle \langle A^{\top}e_{\beta}, u\rangle
=\langle u, A u\rangle.
\]
Let $A: K\rightarrow \mathcal{L}(V)$ be the linear diffusion coefficient, identified by
\begin{align}\label{eq:AQ}
\langle u, A(x) v\rangle=\langle x, Q(u,v) \rangle, \quad u,v\in V,
\end{align}
where $Q: V \times V \to V$ is the symmetric bilinear function in Definition \ref{th: main theorem} \ref{eq:lin_diffusion1}.
Next, we introduce the following linear operator acting on $C^2_b(K)$-functions\footnote{We mean elements of $f\in C^2_b(V)$, restricted to $K$.}
\begin{equation}
\begin{split}\label{eq: generator}
\mathcal{A^\sharp}f(x)&=\frac{1}{2}\text{Tr}\left(A(x) \left(\frac{\partial}{\partial x}\otimes \frac{\partial}{\partial x}\right)\right)f(x)+\left\langle b+B(x), \nabla f(x)\right\rangle
-(c+\langle\gamma,x\rangle)f(x)\\&+\int_{K}
\left(f(x+\xi)-f(x)\right)m(d\xi)+\int_{K}\left(f(x+\xi)-f(x)-\langle
\chi(\xi), \nabla f(x)\rangle\right)\langle x, \mu(d\xi)\rangle.
\end{split}
\end{equation}
A straightforward check reveals that for any $u\in K^*$, the functions $f_u(x)=\exp(-\langle u,x\rangle)$ satisfy
\begin{equation}\label{eq: identif}
\mathcal A^\sharp f_u(x)=f_u(x)(-F(u)-\langle x,R(u)\rangle)=\mathcal A f_u(x),
\end{equation}
where the last identity holds by virtue of Proposition \ref{prop action exp}. Let now $\mathcal S_+$ be the restriction of rapidly decreasing $C^\infty$ functions on $V$ to $K$. Using \eqref{eq: identif} and similar arguments as in the proof of \cite[Proposition 5.9]{cuchiero}\footnote{The argument hinges on
a density argument (see \cite[Proof of Theorem B.3 ]{cuchiero}), namely that the linear hull of $f_u(x)$, $u\in (K^*)^\circ$, is dense in the space of rapidly decreasing functions $\mathcal S_+$ on $K=S_d^+$,
the cone of $d\times d$ positive semi-definite matrices. An inspection of this density argument reveals that it applies to the general cones $K$ of this paper.},
\begin{proposition}
$\mathcal S$ is part of the domain of the infinitesimal generator $\mathcal A)$ and for any $f\in \mathcal S_+$, we have $\mathcal Af(x)=\mathcal A^\sharp f(x)$.
\end{proposition}

For the remainder of this section, we identify $V=\mathbb R^n$ to be able to introduce the spaces $\mathcal P_k$ of polynomials in the variables $x_1,\dots, x_n$, with real coefficients of order $\leq k$\footnote{The orders of polynomials do not depend on the choice of basis.}.
\begin{theorem}\label{thm: cons and pol}
Let $m\geq 2$. Suppose $c=0$, $\gamma=0$, and for some $x_0\in K^\circ$,
\begin{equation}\label{eq: moment condsuper}
\int_{K\setminus\{0\}}\|\xi\|^m (m(d\xi)+\langle x_0,\mu(d\xi)\rangle<\infty.
\end{equation}
Then $X$ is a conservative process with infinitesimal generator $\mathcal A$
acting on $\mathcal S_+$ as
\begin{align}\label{eq: generatorx}
\mathcal{A^\sharp}f(x)&=\frac{1}{2}\text{Tr}\left(A(x) \left(\frac{\partial}{\partial x}\otimes \frac{\partial}{\partial x}\right)\right)f(x)\\\nonumber&\qquad+\left\langle b+\widetilde B(x), \nabla f(x)\right\rangle
+\int_{K}
\left(f(x+\xi)-f(x)-\langle \nabla f(x),\xi\rangle\right)(m(d\xi)+\langle x, \mu(d\xi)),
\end{align}
where $\widetilde B x= B x+\int_{K}
(\xi-\chi(\xi))\langle x, \mu(d\xi)\rangle$
Furthermore, $X$ is an $m$-polynomial process, that is, for any $p\in\mathcal P_m$, any $x\in K$ and for all $t\geq 0$
we have
\[
\mathbb E^x[p(X_t)]<\infty
\]
and the extended generator of $X$ reduces to a (finite dimensional) vector space endomorphism 
$\mathcal A: \mathcal P_m\rightarrow\mathcal P_m$ given by (the pointwise limit)
\[
\lim_{t\downarrow 0}\frac{\mathbb E^x[p(X_t)]-p(x)}{t}=\mathcal A^\sharp p(x),
\]
where $\mathcal A^\sharp$ is defined in \eqref{eq: generatorx}.
\end{theorem}
\begin{proof}
Since $c=0$ and $\gamma=0$ and by \eqref{eq: moment condsuper}, $X$ is conservative due to Proposition \ref{prop cons}. Furthermore, Condition \eqref{eq: moment condsuper} implies $\int_{\|\xi\|\geq 1}
\|\xi\|\langle x, \mu(d\xi)\rangle<\infty$. Therefore the modified drift $\widetilde B =\widetilde B -\int_{K}
(\xi-\chi(\xi))\langle x, \mu(d\xi)\rangle$ is well-defined (since finite), and the special form of the generator \eqref{eq: generatorx} holds.

The second part can be proved by using \cite[Theorem 2.15]{cuchiero2012polynomial} and the realization of $X$ as
a $K$--valued c\`adl\`ag semimartingale (using, e.g., the canonical realization).
\end{proof}
We are prepared to state and prove the following ergodicity result:
\begin{theorem}
\label{prop:fldrift_trace}
Let $X$ be a an affine process on $K$ with $c=0$, $\gamma=0$ and let $m\geq 2$. Assume that
\begin{enumerate}
\item \label{irraperth2} $X$ is $\nu$-irreducible with the support of $\nu$ having non-empty interior and aperiodic.
\item \label{as3th2} There exists $\eta_0\in \mathbb (K^*)^\circ$ such that
\begin{equation}
B^\top \eta_0 + \int_{K\setminus\{0\}} \langle\eta_0,\xi-\chi(\xi)\rangle \mu(d\xi) \in - (K^*)^\circ.
\end{equation}
\item The jump measures have finite $m$--th moment, that is, there exists $x_0\in K^\circ$ such that
\begin{equation}\label{eq:ass_trace}
\int_{K}\|\xi\|^m (m(d\xi)+\langle x_0,\mu(d\xi)\rangle)<\infty.
\end{equation}
\end{enumerate}
Then $X$ is a conservative process. Furthermore, it is geometrically ergodic, positive Harris recurrent and the stationary distribution $\pi$ has finite $m$--th moment, that is,
$\int_{K} ( \|\xi\|^m+1)\pi(d\xi) < \infty$.
\end{theorem}
\begin{remark}
In the below proof, we rely on the theory of \cite{cuchiero2012polynomial} which requires $m\geq 2$ in \eqref{eq:ass_trace}. However, we conjecture that a similar result holds using the first moment.
\end{remark}
\begin{proof}[Proof of Theorem \ref{prop:fldrift_trace}]
First, by Theorem \ref{thm: cons and pol}, $X$ is conservative and $m$--polynomial, and its extended generator is of the form \eqref{eq: generatorx}.

For $\eta\preceq \eta_0$ define the functions $H_\eta: K\rightarrow [1,\infty)$, $H_\eta(x):=\langle \eta,x\rangle^m+1$, which by construction
have range in $[1,\infty)$. Let $D_vf(x)$ be the Fr\'echet differential of a function $f$ at a point $x$ in the direction $v$. We clearly have
\begin{align*}
D_u H_{\eta_0}(x)&=m\langle \eta_0,x\rangle^{m-1}\langle\eta_0,u\rangle,\\
D^2_{(u,v)}H_{\eta_0}(x)&=m(m-1)\langle \eta_0,x\rangle^{m-2}\langle \eta_0,u\rangle\langle \eta_0,v\rangle=O(\|x\|^{m-2}).
\end{align*}
Therefore, a straightforward computation yields
\[
\mathcal A H_{\eta_0}(x)=m\langle \eta_0,x\rangle^{m-1}\left(\langle x,B^\top\eta_0+\int_{K\setminus \{0\}}\langle \eta_0,\xi-\chi(\xi)\rangle\mu(d\xi)\right)+O(\|x\|^{m-1}).
\]
Moreover, by assumption \ref{as3th2} and the fact that $\eta_0\in (K^*)^\circ$, there exists $\lambda>0$ such that 
\[
\widetilde B^\top\eta_0=B^\top\eta_0+\int_{K\setminus \{0\}}\langle \eta_0,\xi-\chi(\xi)\rangle\mu(d\xi)\preceq -\lambda \eta_0.
\]
Hence, 
\[
\mathcal AH_{\eta_0}(x)\leq -\lambda m \langle x,\eta_0\rangle^m+O(\| x\|^{m-1}).
\]
Thus, there exists $d\in\mathbb R$ such that for sufficiently large $r$, we have for the compact ball $B=\{x\in K\mid \|x\|\leq r\} $,
\[
\mathcal A H_{\eta_0}(x)\leq -c H_{\eta_0}(x)+d \cdot 1_{K},
\]
for some $c<\lambda m$. By Assumption \ref{irraperth2} and Proposition \ref{prop:compactpetite}, every compact set is a petite set, whence $(\mathcal A, H)$ satisfies the Foster-Lyapunov drift condition. By Theorem \ref{dmtergodicity}
we may conclude that $X$ is $H$--uniformly ergodic.

Harris recurrence and the finiteness of $m$ moments is proved along the corresponding lines of the Proof of Theorem \ref{prop:FLdrift_expo}.
\end{proof}
\section{Aperiodicity and irreducibility}\label{sec: irr}
The main statements of the previous Section, Theorems \ref{prop:FLdrift_expo} and \ref{prop:fldrift_trace} assume away aperiodicity and irreducibility of the underlying affine process, since the focus was on establishing the needed Foster-Lyapunov drift condition. Establishing (Lebesgue-)irreducibility is typically linked to showing absolute continuity of transition probabilities and usually even to establishing strict positivity of the density. This is a research topic on its own. In this section, we show first that for a large class of diffusion processes, adding finite activity jumps
does not change aperiodicity and irreducibility. As for diffusions (see Remark \ref{rem:diffdens}), the existence of (strictly positive) transition densities is much better understood than for the general affine case with jumps, this immediately establishes Lebesgue-irreducibility and aperiodicity for a large set of affine processes (purely diffusive or with finite activity jumps) on irreducible symmetric cones. Finally , we discuss some finite activity pure jump cases where irreducibility and aperiodicity can be obtained.

For the entire section, we work with the c\`adl\`ag modification of the canonical realization of $X\equiv (X_t)_{t\geq 0}$, for each initial law $\mathbb P^x$. Moreover, we assume that $c=0$, $\gamma=0$, and the jump-behavior is of compound Poisson type, in the sense that the jump measures $m,\mu$ are finite and $\mu$ admits a first moment,
\begin{align}
&m(K)<\infty,\quad \mu(K)\in K^*,\\\label{eq: estimate loc Lip}
&\int_{\|\xi\|\geq 1}\|\xi\|\mu(d\xi)\in K^*.
\end{align}
This implies the following: 
\begin{itemize}
\item for any $x\in K$, $\int_{\|\xi\|\geq 1}\|\xi\|\langle x,\mu(d\xi)\rangle<\infty$,
and therefore $X$ is conservative. Note that $\int_{\|\xi\|\geq 1}\|\xi\|\langle x,\mu(d\xi)\rangle<\infty$ is also implied by the assumptions of Theorems \ref{prop:FLdrift_expo} and \ref{prop:fldrift_trace}.
\item The process $X$ has finite activity jumps, and therefore $\sum_{s\leq t}\Delta X_s:=\sum_{s\leq t} (X_s-X_{s-})$ is well-defined.
\item We may take $\chi\equiv 0$ as truncation function. 
\end{itemize}
By enlarging the state-space, we can disentangle the process from its discontinuous component in a tractable way:
\begin{lemma}\label{lem: affine technology}
For each $x\in K$, $y\in K$, the process $(X_t, Y_t)_{t\geq 0}$, where $Y_t:=y+\sum_{s\leq t}\Delta X_s$,
is an affine process on $K\times K$, with transition probability
\[
\mathbb P^{x,y}[X_t\in A,Y_t\in B]=\mathbb P^x[X_t\in A, y+\sum_{s\leq t}\Delta X_s\in B].
\] 
Its Laplace transform for any $u\in (K^*)^\circ$, $v\in K^*$ is given by
\begin{equation}\label{AFFET}
\mathbb E^{x,y}[e^{-\langle u, X_t\rangle-\langle v, Y_t\rangle}]=e^{-\phi(t,u,v)-\langle \psi(t,u,v),x\rangle-\langle y,v\rangle},
\end{equation}
where $\phi=\phi(t,u,v)$, $\psi=\psi(t,u,v)$ are the unique, global solutions of the generalized Riccati differential equations
\begin{align*}
\partial_t \phi&= \langle b, \psi \rangle - \int_{K \setminus\{0\}} \left(e^{-\langle \psi, \xi \rangle-\langle v,\xi\rangle }-1 \right) m(d\xi), \\
\partial_t \psi&= -\frac{1}{2} Q(\psi,\psi) + B^\top \psi- \int_{K \setminus\{0\}} \left( e^{-\langle \psi, \xi \rangle -\langle v,\xi\rangle}-1\right) \mu(d\xi),
\end{align*}
with initial conditions 
\[ 
\phi(0,(u,v))=0,\quad \psi(0,(u,v))=u\in (K^*)^\circ.
\]
\end{lemma}
\begin{proof}
Due to \eqref{eq: estimate loc Lip}, the map
\[
R(u):= -\frac{1}{2} Q(u,u) + B^\top u- \int_{K \setminus\{0\}} \left( e^{-\langle u, \xi \rangle -\langle v,\xi\rangle}-1\right)\mu(d\xi)
\]
is locally Lipschitz on $(K^*)^\circ$.  Therefore, a unique local solution exists, and we denote by $t_+(u,v)>0$ the explosion time of $\psi(t,u,v)$.

As the jumps of $X$ are of finite total variation, we can write $X=X^c+\sum_{s\leq t} \Delta X_s$, with the continuous part $X^c$. Clearly, the components of
the extended process $(X,Y)$ jump simultaneously and have jumps of equal size ($\Delta X_s=\Delta Y_s$, almost surely), hence for a  bounded, continuous function $f(t,x,y)$, the process
\begin{align*}
&\sum_{s\leq t}\left(f(t,X_{s-}+\Delta X_s, Y_{s-}+\Delta Y_s)-f(t,X_{s-},Y_{s-}\right)\\\qquad&-\int_0^t \left( f(t,X_{s-}+\xi, Y_{s-}+\xi)-f(t,X_{s-},Y_{s-} \right)(m(d\xi)+\langle \mu(d\xi), X_{s_-}\rangle)ds
\end{align*}
is a uniformly integrable martingale. Therefore, by It\^o's formula, for any $t<t_+(u,v)$ the process
\[
e^{-\phi(t-s)-\langle \psi(t-s),X_s\rangle-\langle v, Y_s\rangle},\quad 0\leq s\leq t
\]
is a unformly integrable martingale and thus the affine transform formula \eqref{AFFET} holds for $t<t_+(u,v)$. Since $X$ is an
affine process, by assumption, the generalized Riccati differential equations
\begin{align*}
\partial_t \phi^0&= \langle b, \psi^0 \rangle - \int_{K \setminus\{0\}} \left(e^{-\langle \psi^0, \xi \rangle}-1 \right) m(d\xi), \\
\partial_t \psi^0&= -\frac{1}{2} Q(\psi^0,\psi^0) + B^\top \psi^0- \int_{K \setminus\{0\}} \left( e^{-\langle \psi^0, \xi \rangle}-1\right) \mu(d\xi),
\end{align*}
with initial conditions 
\[ 
\phi^0(0,u)=0,\quad \psi^0(0,u)=u
\]
have global solutions for any $u\in K^*$, and, due to the local Lipschitz property mentioned above, the solutions are unique for $u\in (K^*)^\circ$. Note, $\psi^0(t,u)$ remains in $(K^*)^\circ$ for all times $t\geq 0$, due to \cite[Lemma 3.2]{cuchiero2016affine}. Let $\preceq$ be the partial order induced by the dual cone. Furthermore, since $v\in K^*$, we have
\[
\partial_t\psi(t,u,v)\succeq -\frac{1}{2} Q(\psi,\psi) + B^\top \psi- \int_{K \setminus\{0\}} \left( e^{-\langle \psi, \xi \rangle} -1\right) \mu(d\xi).
\]
Hence for $u\in (K^*)^\circ$ and $v \in K^*$, we have by Volkmann's comparison result \cite{volkmann1973invarianz} that
$\psi(t,u,v)\succeq \psi^0(t,u)$ for any $t<t_+(u,v)$. Due to the afore mentioned property $\psi^0(t,u)\in (K^*)^\circ$, $\psi(t,u,v)$ does not touch the boundary of the dual cone $K^*$ in finite time. We conclude that if $t_+(u,v)<\infty$, then explosion in norm occured as $t\uparrow t_+(u,v)$. But this is impossible, as it would imply, due to
\eqref{AFFET}, that the Laplace transform $\mathbb E^{x,y}[e^{-\langle u, X_t\rangle-\langle v, Y_t\rangle}]=0$
for $t=t_+(u,v)$. Hence we conclude that $t_+(u,v)=\infty$ and the affine transform formula \eqref{AFFET} holds for all $t\geq 0$,
$u\in (K^*)^\circ$ and $v \in K^*$.
\end{proof}
The next statement shows the intuitive fact that before a process jumps, it cannot be distinguished from
a killed diffusion process $\hat X$ with the same characteristics as $X$, except having vanishing jump measures $m$ and $\mu$:
\begin{proposition}\label{prop: before jumps}
Suppose $\hat X$ is an affine processes with parameters $\hat Q=Q,\hat B=B,\hat b=b$, but $\hat m(d\xi)$ and $\hat\mu(d\xi)$ vanish. Then for any $t>0$, $x \in K$ and any $ A \in \mathcal B(K)$ it holds that
\begin{equation}
\mathbb P^x [X_t \in A, \tau >t] = \hat{\mathbb E}^x\left(e^{-\int_0^t (l+ \langle \Lambda, \hat X_s \rangle ) ds} 1_{\{\hat X_t\in A\}}\right),
\label{eq:Xandtau}
\end{equation}
where $l:=m(K)\geq 0$ and $\Lambda=\mu(K)\in K^*$ and $\tau=\inf\{s\geq 0\,|\,\Delta X_s\neq 0\}$.
\end{proposition}
\begin{proof}
In order to show equation \eqref{eq:Xandtau} we have to show that the images of the measures $1_{\{\tau > t\}}d\mathbb P^x$ and $e^{-\int_0^t (l+ \langle \Lambda, \hat X_s \rangle ) ds} d\hat{\mathbb P}^x$ under $X$ and $\hat X$, respectively, are equal. But for that it is enough to prove the identity of the Laplace transformations, i.e. that for any $t\geq 0$, $x\in K$
and $u \in -K^*$,
\begin{equation*}
\mathbb E^x[e^{\langle u, X_t \rangle} 1_{\{\tau >t\}} ]= \hat{\mathbb E}^x [ e^{-\int_0^t (l+\langle \Lambda, \hat X_s\rangle ) ds + \langle u, \hat X_t \rangle}].
\end{equation*}

We consider the $K^2$-valued process $(X,Y)$ as in Lemma \ref{lem: affine technology}. Clearly,
\[
\mathbb P^x[X_t \in A, \tau >t]=\mathbb P^x[X_t\in A, \sum_{s\leq t}X_s=0]=\mathbb P^x[X_t\in A, Y_t=y].
\]
For a random variable $Z \ge 0$ it holds that $\mathbb P[Z=0]=\lim_{w\to -\infty } \mathbb E[e^{wZ}]$.
Applying this trick to $Z:=\langle w_0,Y_t\rangle$ for $y=0$, for some $w_0\in (K^*)^\circ$ we get $\mathbb P^x[\tau >t] = \mathbb P^x[Y_t=0] = \lim_{w\to -\infty } \mathbb E[e^{w\langle v_0,Y_t\rangle}]$.

Setting $y=0$ in Lemma \ref{lem: affine technology} we obtain by continuity for $w \to - \infty$
\begin{align*}
\partial_t \phi^\infty &= \langle b, \psi^\infty \rangle + l, \\
\partial_t \psi^\infty &= -\frac{1}{2} Q( \psi^\infty,\psi^\infty) + B^\top \psi^\infty + \Lambda,
\end{align*}
subject to
\[ 
\phi^\infty(0)=0, \quad \psi^\infty(0)=u.
\]
Note that
\[
e^{\phi^\infty(t,u) + \langle x, \psi^\infty(t,u)\rangle} = 
\hat{\mathbb E}^x(e^{\langle u, X_t^*\rangle} ds )= 
\hat{\mathbb E}^x(e^{\langle u, \hat X_t\rangle} e^{-\int_0^t (l + \langle \Lambda, \hat X_s \rangle ) ds }),
\]
where $X^*$ is an affine process with parameters 
$Q,b,B,c=l,\gamma=\Lambda$ but with vanishing jump measures $\hat m(d\xi),\,\hat{\mu}(d\xi)$.

Summarizing, we have
\begin{align*}
\mathbb E^x[ e^{\langle u, X_t \rangle} 1_{\{ \tau > t \}}] &= \lim_{w \to - \infty} e^{\phi(t,(u,wv_0)) + \langle x, \psi(t,(u,wv_0))\rangle}\\& = e^{\psi^\infty(t,u) + \langle x, \psi^\infty(t,u)\rangle}= \hat{\mathbb E}^x[e^{\langle u, \hat X_t\rangle} e^{-\int_0^t (l + \langle \Lambda, \hat X_s \rangle ) ds }].
\end{align*}
\end{proof}
The following essentially states that adding finite activity jumps to a diffusion with strictly positive transition densities maintains 
irreducibility and aperiodicity.
\begin{proposition}\label{prop: irreducible}
Suppose $\hat X$ is an affine processes with parameters $\hat Q=Q,\hat B=B,\hat b=b$, but $\hat m(d\xi)$ and $\hat\mu(d\xi)$ vanish.

Suppose for any $t>0,x\in K$, the transition probabilities $\hat p_t(x,d\xi)$ of $\hat X$ have a Lebesgue density which is strictly positive on $K^\circ$. Then both $X$ and $\hat X$ are aperiodic and irreducible (with respect to the Lebesgue measure $\lambda$ on $K$).
\end{proposition}
\begin{proof}
Let $\tau:= \inf \{ t>0 \mid \Delta X_t \neq 0\}$ be again the first jump time of $X$. By the law of total probability, for any Borel set $A$, and any $t>0$ and for any $x\in K$ we have
\[
\mathbb P^x[X_t \in A]\geq \mathbb P^x[X_t \in A, \tau >t]=\mathbb P^x[X_t \in A\mid \tau >t]\cdot \mathbb P^x[\tau >t].
\]
We show first that for any $A\in \mathcal B(K)$ with $\lambda(A) >0$ and any $t>0$ it holds that
\begin{equation}\label{eq: crucial}
\mathbb P^x[X_t \in A, \tau >t]>0, \quad\forall x\in K. \\
\end{equation}
By assumption, for any $A \in \mathcal B(K)$ with $\lambda(A) >0$ it holds that $\hat{\mathbb P}^x[\hat X_t \in A]>0$ $\forall x \in K$ (note that the boundary of $K$ is a null set, see e.g. \cite{Lang1986}). Fix $t>0, x \in K$ and $A \in \mathcal B(K)$ with $\lambda(A) >0$. Let $\varepsilon := \frac 1 2 \mathbb P^x[\hat X_t \in A]$. Due to the continuity of the path of $X$ there exists $\eta_\varepsilon >0$ such that 
\[ 
\mathbb P^x [\sup_{0 \leq s \le t} \langle \hat X_s,\Lambda\rangle \leq \eta_\varepsilon] \geq 1- \varepsilon. 
\]
Thus for $B_\varepsilon := \{ \sup_{0 \leq s \leq t} \langle X_s, \Lambda\rangle \leq \eta_\varepsilon \}$, we obtain
\[ 
\mathbb P^x [(\hat X_t \in A) \cap B_\varepsilon]\geq 2\varepsilon + 1 - \varepsilon - \mathbb P^x [\hat X_t \in A \cup B_\varepsilon) \geq \varepsilon . 
\]
Using $\mathbb P^x [\hat X_t \in A] \geq \mathbb P^x [(\hat X_t \in A) \cap B_\varepsilon]$ we can conclude
\begin{align*}
\mathbb P^x [X_t \in A, \tau > t]&= \mathbb E^x [e^{- \int_0^t (l + \langle \Lambda, \hat X_s \rangle ) ds} 1_{\{ \hat X_t \in A \}})]\\
&\geq e^{- (l + \eta_\varepsilon)t} \mathbb E^x[1_{\{(\hat X_t \in A) \cap B_\varepsilon\}} ]\\
& \geq e^{- (l + \eta_\varepsilon)t} \varepsilon >0,
\end{align*}
and estimate \eqref{eq: crucial} is shown.

Thus, for any $A\in\mathcal B(K)$ for which $\lambda(A)>0$, it follows that
$\hat p_t(x,A)>0$ and $\hat p_t(x,A)>0$ and thus the processes are $\lambda$--irreducible.

Finally, we show aperiodicity, i.e., that there exists a small set $\mathcal C$ and a time $T>0$ 
such that for any $t>T$,
\begin{equation}\label{prop pos}
\mathbb P^x[X_t \in \mathcal C] >0, \quad x\in \mathcal C.
\end{equation}
By strict positivity of the transition densities (see first part), the compact set $\mathcal C:=\{\xi\in K\mid \|\xi\|\leq 1\}$ satisfies \eqref{prop pos}; it therefore suffices to show that $\mathcal C$ is small for $X$: By Proposition \ref{prop:compactpetite}, every compact set is petite for any $T$--skeleton chain. Furthermore, by $\lambda$--irreducibility of $X$, every $T$-skeleton chain is irreducible and thus also aperiodic, see \cite[Proposition 1.2]{TT1979}. Therefore every petite set is small for every $T$-skeleton chain, see \cite[Proposition 5.5.7]{MT}. We have thus established that $\mathcal C$ is small for any $T$--skeleton chain ($T>0$). In other words, there exists $n\in\mathbb N$ and and a non-trivial measure $\nu_n(d\xi)$
such that for all $x\in \mathcal \mathcal C$ we have $p_{nT}(x,\cdot)=\int p_{kT}(x,\dot) a(dt)\geq \nu_n(\cdot)$, where $a(dt)=\delta_{nT}(dt)$. Hence, by definition, it is small for the continuous-time Markov process $X$ as well.\footnote{Besides, any small set is petite (as it is petite by definition, for a degenerate sampling distribution $a(dt)$).}\footnote{See the paragraph after Equation (10) in \cite{DMT1995}.}

Since $\hat X$ is just a process with the same characteristics as $X$, but with zero jumps, aperiodicity also holds for $\hat X$, and we are done.
\end{proof}
\begin{remark}\label{rem:diffdens}
\begin{itemize}
\item For a large class of diffusions on irreducible symmetric cones (the so-called ``Bru processes with linear drift"), the existence of a strictly positive Lebesgue density for each $t>0$ and $x\in K$ is characterized in terms of a standard controllability condition \cite[Section 5.1.4]{cuchiero2016affine}. In fact, in this case for each $t>0$, the process $X_t$ is non-centrally Wishart distributed, and in this case the transition density is explicitly given, is supported on $K^\circ$ and is strictly positive thereon. We provide in the next section as a special case the Wishart processes on the cones of positive semi-definite $d\times d$ matrices (in which the controllability condition is Theorem \ref{cor: wish} {\ref{cond: hypo}}).
\item
 In the pure jump case, \cite{Simon2011} establishes under a comparable controllability condition under infinite activity jumps the absolute continuity of the transition probabilities of an Ornstein-Uhlenbeck (OU) process. However, irreducibility holds under weaker conditions for the OU process, see e.g., \cite{masuda2004} and \cite{PigorschetStelzer2007b} for examples of non-degenerate/absolutely continuous stationary distributions in the context of cones).
\end{itemize}
\end{remark}
We now turn to to the pure-jump case. So $Q=0$ in addition to the assumptions made so far.
The following sufficient conditions for irreducibility and aperiodicity are inspired by similar results of \cite{StelzerVestweber2017} for a non-linear SDE in the cone of positive semi-definite matrices.
\begin{proposition}\label{prop:purejumpirr}
Let $X$ be a pure jump affine process as described above. If
\begin{enumerate}
\item \label{cond1}$m$ or $\langle \eta,\mu(\cdot)\rangle$ for one $\eta\in K^\circ$ have an absolutely continuous component with a strictly positive density on $K^\circ$,
\item $e^{Bt}(K^\circ)=K^\circ$ for all $t\in \R_+$ and $\tau(B)<0$,
\item \label{cond3}$m(K^\circ)>0$ or $b\in K^\circ$ provided $m$ has no absolutely continuous component with a strictly positive density on $K^\circ$,
\end{enumerate}
then $X$ is irreducible with respect to the Lebesgue measure on $B^{-1}b+K$ and aperiodic.
\end{proposition}
\begin{proof} 
Since aperiodicity and irreducibility are properties of the law of the process, we can assume that $X$ is constructed as follows on an appropriate probability space: Let $(\sigma_j)_{j\in\mathbb{N}}$ be an iid sequence of $Exp(1)$ distributed random variables.
Set
\begin{align*}
\tau_1&=\rho_1= \inf\left\{z>0:\int_{0}^z \left(m(K)+\left\langle e^{Bt} x+\int_{0}^te^{B(t-s)}bdt,\mu(K)\right\rangle\right)dt \geq \sigma_1\right\}\\
X_t&=e^{Bt}x+\int_0^te^{B(t-s)}bds\, \mbox{ for } t\in [0,\tau_1).
\end{align*}
Condition \ref{cond3} implies that $\tau_1<\infty$ almost surely. Next draw a $K$-valued random variable $\Delta X_{\tau_1}$ with distribution $(m(\cdot)+\langle X_{\tau_1-},\mu(\cdot) \rangle)/(m(K)+\langle X_{\tau_1-},\mu(K)\rangle)$ and otherwise independent of $(\sigma_j)_{j\in\mathbb{N}}$. Set 
\[
X_{\tau_1}=e^{B\tau_1}x+\int_0^{\tau_1}e^{B(\tau_1-s)}bds+\Delta X_{\tau_1}.
\]
For $j=2,3,\ldots$ apply the following procedure iteratively. Set
\begin{align*}
\rho_j&= \inf\left\{z>0:\int\limits_{\tau_{j-1}}^{\tau_{j-1}+z} \left(m(K)+\left\langle e^{Bt} x+\int_{0}^te^{B(t-s)}bdt+\sum_{k=1}^{j-1}e^{B(t-\tau_k)}\Delta X_{\tau_k},\mu(K)\right\rangle\right)dt \geq \sigma_j\right\},\\
\tau_j&=\tau_{j-1}+\rho_j,\\
X_t&=e^{Bt}x+\int_0^te^{B(t-s)}bds+\sum_{k=1}^{j-1}e^{B(t-\tau_k)}\Delta X_{\tau_k}\, \mbox{ for } t\in (\tau_{j-1},\tau_j).
\end{align*}
Condition \ref{cond3} implies that $\rho_j<\infty$ a.s. and so $\tau_j<\infty$ a.s. 
Then draw a $K$-valued random variable $\Delta X_{\tau_j}$ with distribution $(m(\cdot)+\langle X_{\tau_j-},\mu(\cdot) \rangle)/(m(K)+\langle X_{\tau_j-},\mu(K)\rangle)$ and otherwise independent of $\{(\sigma_j)_{j\in\mathbb{N}}, (\Delta X_{\tau_k})_{k=1,..j-1}\}$.
Set 
\[
X_{\tau_j}=e^{B\tau_j}x+\int_0^{\tau_j}e^{B(\tau_j-s)}bds+\sum_{k=1}^{j-1}e^{B(\tau_j-\tau_k)}\Delta X_{\tau_k}+\Delta X_{\tau_j}.
\]
Obviously this gives a process $X$ with the correct semimartingale characteristics (see \cite{Jacodetal2003} for a comprehensive treatment). 

Furthermore, by an adaption of the proof of \cite[Proposition 5.3]{cuchiero}\footnote{ Set $Q=0$ and replace $S_d^{++}$ by $(K^*)^\circ$ in \cite[Proposition 5.3]{cuchiero}.}, the associated generalized Riccati differential equations \eqref{eq:RiccatiF}--\eqref{eq:RiccatiR} have a unique global solution
for any initial data $u\in (K^*)^\circ$. We introduce the filtration $\mathcal F:=(\mathcal F_t)_{t\geq 0}$, where $\mathcal F_t:=\sigma(X_s\mid s\leq t)$. By It\^o's formula and the generalized Riccati differential equations, we find that for any $t>0$, the process $e^{-\Phi(t-s,u)-\langle \Psi(t-s,u),X_s\rangle}$ is a uniformly integrable martingale on $[0,t]$, hence we conclude
the affine transform formula holds, that is,
\[
\mathbb E[e^{-\langle u,X_t\rangle}\mid \mathcal F_s]=e^{-\Phi(t-s,u)-\langle \Psi(t-s,u),X_s\rangle}, \quad 0\leq s<t.
\]
Therefore, by a similar argument as in \cite[Section 8]{filipovic2009affine}, the Laplace
transforms of the finite-dimensional distributions of $X$ can be computed explicitly in terms of the functional characteristics $\Phi$, $\Psi$.\footnote{Due to a density argument, it suffices to consider Laplace variables on $(K^*)^\circ$.} Since the finite-dimensional distributions uniquely characterize the process in law, $X$ is an affine process with parameters $b,B,m(d\xi)$ and $\mu(d\xi)$, as desired.

First, we  establish the claimed irreducibility: Observe that for any $x\in K$,
\begin{equation}
\lim_{t\to \infty}\left(e^{Bt}x+\int_0^te^{B(t-s)}bds\right)=B^{-1}b \in K.
\end{equation}
So from wherever we start we get eventually ``down'' towards $B^{-1}b$, as long as no jumps occur. The rest of the proof is about rigorously establishing that from there we can reach any relevant set with positive probability at any large enough time.  Fix some $x\in K$ and a measurable set $A\subseteq B^{-1}b+K$ with non-vanishing Lebesgue measure. It suffices to show that there exists a $T>0$ such that
\[
\mathbb P^x[X_t\in A]>0\,\forall t> T.
\]
The above construction of $X$ and Condition \ref{cond3} imply that $\mathbb P^x[X_t\in K^\circ]>0$ for all $x\in K$ and $t>0$. Therefore, we can assume w.l.o.g. that $x\in K^\circ$ and thus $\mathbb P^x [X_t\in K^\circ\,\forall \, t\in \mathbb R^+]=1$. Pick an $\eta\in K^\circ$ and an $\epsilon >0$ such that $\lambda(\widetilde A)>0$ for $\widetilde A:=A\cap(\epsilon \eta+B^{-1}b+K)$. Now we can choose a $T>0$ such that $e^{Bt}x+\int_0^te^{B(t-s)}bds\preceq\epsilon \eta+B^{-1}b$ for all $t\geq T$. 
Pick an arbitrary $t>T$. 
\begin{align*}
\mathbb P^x[X_t\in A]&\geq \mathbb P^x[\{X_t\in \widetilde A\}\cap \{\tau_1\in (T,t]\}\cap\{\tau_2>t\}]\\&=\int_T^t \mathbb P^x\left[\left.\Delta X_{\tau_1}\in e^{-B(t-s)}\left(\widetilde A- e^{Bt}x-\int_0^te^{B(t-u)}bdu\right)\right| \tau_2>t,\tau_1=s\right]\mathbb P_{\tau_1}^x[ds|{\tau_2>t}]\\&\quad\cdot \mathbb P^x[\tau_2>t].
\end{align*}
Here $\mathbb P_{\tau_1}^x[\cdot|{\tau_2>t}]$ is the conditional law of $\tau_1$ under $\mathbb P^x$ given $\tau_2>t$ and $\mathbb P^x\left[\left.\cdot \right| \tau_2>t,\tau_1=s\right] $ a version of the regular conditional probability given $\tau_1=s$. From the above explicit construction we now see that $\mathbb P^x[\tau_2>t]>0$ and that $P_{\tau_1}^x(\cdot|{\tau_2>t})$ is equivalent to the Lebesgue measure on $(T,t]$. Observe that the measures $(\langle \eta,\mu(\cdot)\rangle)_{\eta\in K^\circ}$ are all equivalent. Hence, $\mathcal{M}^+(\eta):=\{M\in \mathcal{B}(K): m(M)+\langle \eta,\mu(M)\rangle>0\}$ (with $\mathcal{B}(K)$ denoting the Borel-$\sigma$-algebra over $K$) does not depend on $\eta\in K^\circ$. Due to Condition \ref{cond1} $\mathcal{M}^+(\eta)\supseteq \{M\in \mathcal{B}(K): \lambda(M)>0\}$. From the explicit construction we see that $\mathbb P^x\left[\left.\Delta X_{\tau_1} \in M\right| \tau_2>t,\tau_1=s\right]>0$ holds for all $M\in\mathcal{M}^+(\eta)$ and $s\in(T,t]$. Due to our choices the set $\widetilde A- e^{Bt}x-\int_0^te^{B(t-u)}bdu$ is in $ K$ and has non-zero Lebesgue measure. As $e^{B(t-s)}$ is a diffeomorphism the standard transformation formula (e.g. \cite[Theorem 1.101]{Klenke2014}), implies that $e^{-B(t-s)}\left(\widetilde A- e^{Bt}x-\int_0^te^{B(t-u)}bdu\right) $ has non-zero Lebesgue measure for all $s\in(T,t]$. Thus it belongs to $M\in\mathcal{M}^+(\eta)$. Hence, $\mathbb P^x\left[\left.\Delta X_{\tau_1}\in e^{-B(t-s)}\left(\widetilde A- e^{Bt}x-\int_0^te^{B(t-u)}bdu\right)\right| \tau_2>t,\tau_1=s\right]>0$ holds for all $s\in (T,t]$ and therefore $\mathbb P^x[X_t\in A]>0$ for all $t>T$ concluding the proof of the irreducibility.

It remains to show the aperiodicity.
Consider the set $\mathcal{C}=B_{1}(B^{-1}b)\cap K$. Arguing as in the proof of Proposition \ref{prop: irreducible} we see that $\mathcal{C}$ is small. Choose $T>0$ such that $\int_0^te^{B(t-s)}bds\in B_{1/2}(B^{-1}b)$ and $\|e^{Bt}x\|<1/2$ for all $x\in \mathcal{C}$ and $t>T$. Hence, $e^{Bt}x+\int_0^te^{B(t-s)}bds \in \mathcal{C}$ and therefore
\[
\mathbb P^x[X_t\in \mathcal{C}]\geq \mathbb P^x[ \tau_1>t]>0\,\, \forall t>T, x\in\mathcal{C}.
\]
Therefore the process is aperiodic.
\end{proof}
\begin{remark}\label{rem:purejumpirr}
\begin{itemize}
\item Similarly to \cite[Corollary 4.7]{StelzerVestweber2017} one can relax the conditions to demanding that $m$ or $\langle \eta,\mu(\cdot)\rangle$ for one $\eta\in K^\circ$ have an absolutely continuous component with a strictly positive density only in a neighbourhood of zero in $K^\circ$. Then $X$ is irreducible with respect to the Lebesgue measure restricted to a neighbourhood of $B^{-1}b$ and aperiodic.
\item If $K$ is the cone of $d\times d$ symmetric positive definite matrices, one can also build the jumps by taking the outer products of jumps in $\R^d$ and the above results remain true provided that they have a strictly positive density on $\R^d$ (in a neighbourhood of zero). This can be shown by taking not only one jump after time $T$ as in the above proof but at least $d$ jumps and arguing similar to \cite[{Theorem 4.6}]{StelzerVestweber2017}. 
\end{itemize}
\end{remark}

\section{Applications and extensions.}\label{sec: appl}
\subsection{Wishart processes}\label{Sec: wishart examples}
Let $V=S_d$ be the symmetric $d\times d$ matrices, endowed with the inner product $\langle x,y\rangle=\text{Tr}(xy)$, where $xy$
denotes the matrix product, and $\text{Tr}$ is the trace of a $d\times d$ matrix. On the positive semi-definite symmetric $d\times d$ matrices $K=S_d^+\subset V$, the Bru processes with linear drift are just the Wishart processes, which are for each initial value $x\in S_d^+$ the global unique weak solution of
\begin{align}
dX_t &= \sqrt{X_t} dW_t Q + Q^\top dW_t^\top \sqrt{X_t} + (X_t \beta + \beta^\top X_t + \delta Q^\top Q)dt,\label{eq:wishartsde} \\
X_0 & = x \in S_d^+,
\end{align}
where $Q,\beta $ are real $d\times d$ matrix (not necessarily symmetric), $W$ is a standard $d\times d$
Brownian motion matrix, and $\delta> d-1$. In this case, the parameters in the sense of Theorem \ref{th: main theorem} are
\begin{itemize}
\item $b=\delta \alpha$, where $\alpha=Q^\top Q\in S_d^+$, and $\delta>d-1$,
\item $Q: V\times V\rightarrow V$, $Q(u,u)=2 u \alpha u$, and
\item $B^\top: V\rightarrow V$ is given by $B^\top(u)=\beta u+u\beta^\top$,
\end{itemize}
while $c$, $\gamma$, $m(d\xi)$ and $\mu(d\xi)$ vanish. 

Under the above conditions \cite{cuchiero} showed that the SDE \eqref{eq:wishartsde} has a unique global weak solution remaining in $S_d^+$. Strengthening the above assumptions to $\delta\geq d+1$ \cite{Mayerhofer2009} showed that a unique strong solution exists which always stays in the interior of the cone provided the initial value is strictly positive definite. Intuitively, the inward pointing drift must be sufficiently large to ensure that the Brownian term cannot move the process through the boundary.

Note that for any $x,u\in S_d^+$ by the definition of the inner product, and the cyclic property of the trace,
\[
\langle B^\top(u),x\rangle=\langle \beta u+u\beta^\top,x)=\langle u,x\beta\rangle+\langle u,\beta^\top x\rangle=\langle u, \beta^\top x+x\beta\rangle,
\]
whence $B(u)=u\beta+\beta^\top u$.

An application of Theorem \ref{thm: linag} to the self-dual cone $K=S_d^+$ is the following:
\begin{corollary}\label{cor: linag sdp}
The following are equivalent:
\begin{enumerate}
\item \label{itemx1} $\tau(\beta)<0$.
\item \label{qmi 1app} $\tau(B)<0$.
\item \label{qmi 2app} There exists $v\in S_d^{++}$ such that $\beta v+v\beta^\top\in -S_d^{++}$.
\end{enumerate}
\end{corollary}
\begin{proof}
By construction, the linear map $B$ is a qmi map on $S_d^+$ with respect to $S_d^+$ (in fact, even $-B$ is quasi monotone increasing). Therefore, Theorem \ref{thm: linag}
gives the equivalent of \ref{qmi 1app} and \ref{qmi 2app}.

It remains to prove that \ref{itemx1} is equivalent to any of the two equivalent statements \ref{qmi 1app} and \ref{qmi 2app}.

By \cite[Theorem 4.45]{Hornetal1991} we have $\sigma(B)\subseteq \sigma(\beta)+\sigma(\beta)$ (as our operator $B$ is acting only on $S_d$), which settles the proof of
\ref{itemx1} $\Rightarrow$ \ref{qmi 1app}.

Let now $\lambda,\mu\in \mathbb{C}$ be two eigenvalues of $\beta^T$ with corresponding eigenvectors $x,y\in \mathbb{C}^d$. Then $\beta^T (xy^*+yx^*)+(xy^*+yx^*)^*\beta=(\lambda+\mu)(xy^*+yx^*)$. As $xy^*+yx^*$ is Hermitian this shows that $B^T$ has the eigenvalue $\lambda+\mu$ as an operator over the Hermitian matrices and thus over the real symmetric matrices. Hence, $\sigma(B)\supseteq \sigma(\beta)+\sigma(\beta)$. This settles the proof of \ref{qmi 1app} $\Rightarrow$ \ref{itemx1}. 
\end{proof}

A consequence of Corollary \ref{cor:expmom}, Proposition \ref{prop: irreducible} and Corollary \ref{cor: linag sdp} is the geometric ergodicity of Wishart processes:
\begin{theorem}\label{cor: wish}
Let $X$ be a Wishart process with parameters $(\delta, \beta, Q)$. Assume that
\begin{enumerate}
\item $\tau(\beta)<0$ (or, any equivalent condition in Corollary \ref{cor: linag sdp}).
\item \label{cond: hypo} The $d \times d^2$ matrix $Q^\top | \beta Q^\top | \ldots | \beta^{d-1} Q^\top$ has maximal rank (equals $d$). 
\end{enumerate}
Then $X$ is geometrically ergodic, positive Harris recurrent and the stationary distribution has some finite exponential moments.
\end{theorem}
\begin{remark}
As the distribution of Wishart processes is exactly known and the stationary distribution has to be a Wishart distribution, it is clear that the stationary distribution has some finite exponential moments. The main implication of our approach is that we get geometric ergodicity and thus exponentially fast convergence in the total variation norm (and actually in an exponential norm).
\end{remark}
\begin{proof}[Proof of Theorem \ref{cor: wish}]
It suffices to show that all assumptions of Corollary \ref{cor:expmom} are satisfied:

By \cite[Theorem 7.1]{mayerhofer2012}, $X_t\mid X_0=x$ admits a density, for each $t>0$. By inspection of the series expansion of the density \cite[eq (5.3)]{mayerhofer2012},
we see that if for some multi-index $\kappa$, the zonal polynomial $C_\kappa$ is strictly positive on $S_d^{++}$, then the density is strictly positive thereon. In fact, by
\cite[Lemma 3.1]{mayerhofer2017} this is the case for $\kappa=(1,\dots,1)$, which is fairly obvious from the fact that in this case $C_\kappa=\det$, because the determinant is invariant under the action of the orthogonal group. By Proposition \ref{prop: irreducible}, $X$ is aperiodic and irreducible and thus satisfies Condition \ref{thm11cor} of Corollary \ref{cor:expmom}.

By Corollary \ref{cor: linag sdp} also Condition \ref{thm11cor} of Corollary \ref{cor:expmom} is satisfied, and the latter corollary concludes the proof.
\end{proof}
Of course, we can extend the results of this section by adding jumps. Consider the process
\begin{align}
dX_t &= \sqrt{X_t} dW_t Q + Q^\top dW_t^\top \sqrt{X_t} + (X_t \beta + \beta^\top X_t + \delta Q^\top Q)dt+ dJ_t, \\
X_0 & = x \in S_d^+,
\end{align}
where $Q,\beta $ are real $d\times d$ matrix (not necessarily symmetric), $W$ is a standard $d\times d$
Brownian motion matrix, and $\delta> d-1$. Furthermore $J_t=\sum_{s\leq t}\Delta X_s$ has absolutely continuous compensator $\nu(X_t,d\xi)dt=\left(m(d\xi)+\langle X_t, \mu(d\xi)\rangle\right)dt$, where $m(d\xi)$ is a finite positive measure on $S_d^+$ satisfying $m(S_d^+)=l$, and $\mu(d\xi)$ is an $S_d^{+}$--valued measure on $S_d^+$ satisfying
$\mu(S_d^+)=\Lambda$. In this case, the parameters in the sense of Theorem \ref{th: main theorem} are
\begin{itemize}
\item $b=\delta \alpha$, where $\alpha=Q^\top Q\in S_d^+$, and $\delta>d-1$,
\item $Q: V\times V\rightarrow V$, $Q(u,u)=2 u \alpha u$, and
\item $B^\top: V\rightarrow V$ is given by $B^\top(u)=\beta u+u\beta^\top$,
\item The jump measures $m(d\xi)$ and $\mu(d\xi)$ as above. 
\end{itemize}
We shall refer to this class of affine jump-diffusions as {\it Wishart processes with jumps}.
\begin{theorem}\label{cor: wish with jumps}
Let $X$ be a Wishart process with jumps. Assume that
\begin{enumerate}
\item \label{pointt1} The eigenvalues of the linear map $L:S_d\rightarrow S_d$ defined by
\[
Lu=\beta u+u\beta^\top+\int_{S_d^+}\tr( \xi u) \mu(d\xi)
\]
have strictly negative real parts.
\item there exists $\eta_0\in S_d^{++} $ such that $\int_{\{ \|x\|\ge 1\}} e^{ \tr(\eta_0 \xi)} (m(d\xi)+\tr(\mu(d\xi),x)) < \infty.
$
\item \label{pointt2} The $d \times d^2$ matrix $Q^\top | \beta Q^\top | \ldots | \beta^{d-1} Q^\top$ has maximal rank (equals $d$). 
\end{enumerate}
Then $X$ is geometrically ergodic, positive Harris recurrent and the stationary distribution has some finite exponential moments.
\end{theorem}
\begin{proof}
Due to Proposition \ref{prop: before jumps}, prior to its first jump, the Wishart proces with jumps $X$ cannot be distinguished from
a Wishart process $\hat X$, killed at the exponentially affine rate $e^{-\int_0^t {l+\tr (\Lambda X_t) }}$. Since $\hat X_t\mid \hat X_0=x$ admits a strict positive density, 
for each $t>0$ (cf.~proof of Theorem \ref{cor: wish}), the process $X$ is aperiodic and irreducible (Proposition \ref{prop: irreducible}). Hence the main assertion follows by an application Corollary \ref{cor:expmom}.
\end{proof}
Using Theorem \ref{prop:fldrift_trace} we similarly get:
\begin{theorem}
Let $X$ be a Wishart process with jumps and $m\geq 2$ a natural number. Assume that
\begin{enumerate}
\item \label{pointt1x} The eigenvalues of the linear map $L:S_d\rightarrow S_d$ defined by
\[
Lu=\beta u+u\beta^\top+\int_{S_d^+}\tr( \xi u) \mu(d\xi)
\]
have strictly negative real parts.
\item The jump measures have finite $m$--th moment, that is, there exists $x_0\in S_d^{++}$ such that
\begin{equation}\label{eq:ass_tracex}
\int_{S_d^+}\|\xi\|^m (m(d\xi)+\tr(x_0,\mu(d\xi)))<\infty.
\end{equation}
\item \label{pointt2x} The $d \times d^2$ matrix $Q^\top | \beta Q^\top | \ldots | \beta^{d-1} Q^\top$ has maximal rank (equals $d$). 
\end{enumerate}
Then $X$ is geometrically ergodic, positive Harris recurrent and the stationary distribution 
$\pi$ has finite $m$--th moment, that is,
$\int_{S_d^+} ( \|\xi\|^m+1)\pi(d\xi) < \infty$.
\end{theorem}

\begin{remark}\label{rem: app dim 1}
For dimension $d=1$, vanishing $\mu(d\xi)$, and $m(d\xi)=c d e^{-d\xi}d\xi$ (exponentially distributed jumps with mean $d$, intensity $c$)
we obtain the corresponding results for the ``Basic affine jump diffusion" (BAJD), \cite[Theorem 5.6 and Theorem 6.3]{jin2016positive}, given by the SDE
\[
dX_t=a(\theta-X_t)dt+\sigma \sqrt{X_t}dW_t+d\hat J_t,\quad X_0\geq 0,
\]
Here $\hat J$ the compensated jump process.

Note that in this case, $\sigma>0$ implies strictly positive transition densities (as for $J=0$ they are non-central chi-square distributed. Alternatively,
{\ref{pointt2}} of Theorem \ref{cor: wish with jumps} also implies strictly positive densities). Furthermore, since $J$ is compensated, the standard assumption $a>0$
implies Condition {\ref{pointt1}} of Theorem \ref{cor: wish with jumps}.

In this special one-dimensional case \cite{JinKremerRuediger2017} could also establish ergodicity under log moments and geometric ergodicity under the finiteness of $\kappa>0$ moments and for $J$ being an infinite activity subordinator, i.e. only having state-independent jumps.
\end{remark}
\subsection{Affine processes on $\mathbb R_+^m$}
In the following, the inner product $\langle x,y\rangle=x^\top y$ is the standard Euclidean scalar product on $V=\mathbb R^m$.

Affine processes with state space $K=\mathbb R_+^m$ are fully characterized by the paper \cite{duffie2003}. In this case, the linear drift map $B^\top:\mathbb R^m\rightarrow\mathbb R^m$ is a $m\times m$ matrix with non-negative off-diagonal elements, hence it is easy to compute the eigenvalues of $B$ to find a good test vector satisfying condition \ref{thm12} of
Theorem \ref{prop:FLdrift_expo}. A corollary is therefore the following ergodicity result:
\begin{theorem}
Let $X$ be an affine process on $\mathbb R_+^m$ with $c=0$, $\gamma=0$. Assume that 
\begin{enumerate}
\item \label{thm11x} $X$ is $\nu$-irreducible with the support of $\nu$ having non-empty interior and aperiodic,
\item \label{thm12x} The eigenvalues of the linear map $B^\top+\int_{\mathbb R_+^m\setminus\{0\}}\langle \xi-\chi(\xi)\rangle \mu(d\xi)$ have strictly negative real parts, and for sufficiently small $w>0$, 
\begin{align*}
&\int_{\{ \|x\|\ge 1\}} e^{ \|w\xi\|} (m(d\xi)+\langle\mu(d\xi),x\rangle < \infty.
\end{align*} 
\end{enumerate} 
Then $X$ is conservative, and there exists $\eta_0\in\mathbb R_{++}^m$ such that for all $\eta\preceq \eta_0$, any $x\in\mathbb R_+^m$ and all $t\geq 0$ we have $\mathbb E^x[e^{\langle \eta,X_t\rangle}]<\infty$.

Furthermore, $X$ is geometrically ergodic, positive Harris recurrent. and for any $\eta\preceq \eta_0$ the stationary distribution $\pi$ has finite $\eta$-exponential moment, i.e. $\int_{K\setminus \{0\}} e^{\langle\eta_0, \xi\rangle} \pi(d\xi) < \infty$.
\end{theorem}
We refrain from stating the obvious analogue with finite $m$-th moments. For conditions ensuring irreducibility and aperiodicity we refer to Remarks \ref{rem:diffdens}, \ref{rem:purejumpirr}, Proposition \ref{prop:purejumpirr} and \cite{ZhangGlynn2018} who discuss the canonical state space at length under a strong non-degeneracy assumption on the diffusive part which is stronger than the controllability assumption. 
\subsection{Affine term-structure models.}
We consider state spaces $K\times\mathbb R^n$ in this section, where $K$ is an $m$-dimensional generating, proper closed convex cone. There is no established theory of affine processes on these spaces,
and also our theory developed in Chapter \ref{sec:affine_ge} does not fall into this setup, as $K\times\mathbb R^n$ is not a proper cone.

Due to \cite[Lemma 7.1]{filipovic2009affine} the instantaneous diffusion of processes on $\mathbb R_+^m\times\mathbb R^n$ can always be brought into block-diagonal form. We take this statement as motivation to make the following assumption throughout. Let $Z:=(X,Y)$ be an affine diffusion process
on $K\times \mathbb R^n$ such that for each $z:=(x,y)\in K\times \mathbb R_+^n$, $X$ is the unique weak solution to
\begin{align}
dX_t&=(b+B(X_t))dt+\Sigma(X_t)dW^X_t,\quad X_0\in K,\\
dY_t&=(c+C(X_t)+D(Y_t))dt+\sigma(X_t)dW^Y_t,\quad Y_0\in \mathbb R^n,
\end{align}
where $W^X$ and $W^Y$ are independent standard Brownian motions of dimensions $m,n$, respectively.
Furthermore $b\in K$, $c\in\mathbb R^n$, and $B$ is a linear map, qmi with respect to $K$. Moreover, $C:\mathbb R^m\rightarrow\mathbb R^n$ and $D:K \rightarrow\mathbb R^n$ are linear maps. The affine structure implies
that $\langle u,\Sigma (x)\Sigma^\top (x)u\rangle=\langle Q(u,u),x\rangle$ as in Theorem \ref{th: main theorem} \ref{eq:lin_diffusion1}, while $\sigma\sigma^\top(x)=E+F x$,
where $E\in S_n^+$ and $F: K\rightarrow S_n^+$ is linear.

For the associated generalized Riccati differential equations, the origin in $ K\times \mathbb{R}^{n}$ is an asymptotically stable equilibrium, if the eigenvalues of the drift matrix
\[
\Xi:=\left(\begin{array}{ll} B & 0\\ C&D\end{array}\right)
\]
(or, equivalently, that of $\Xi^\top$) have strictly negative parts. This can be seen by linearizing the associated generalized Riccati differential equations for $\psi(t,u,v)=(\psi^X,\psi^Y)$ which are given by
\[
\partial_t\psi^X=B^\top \psi^X+C^\top \psi^Y+ \frac{1}{2}Q(\psi^X,\psi^X)+(\psi^Y)^\top F \psi^Y,
\]
whereas
\[
\partial_t \psi^Y=D^\top \psi^Y,
\]
subject to $\psi(t=0,u,v)=(u,v)$.
\begin{theorem}\label{megageil}
Suppose $Z$ satisfies the following:
\begin{enumerate}
\item \label{part1 term} $Z$ is $\nu$-irreducible with the support of $\nu$ having non-empty interior and aperiodic,
\item \label{part2 term} $\tau(\Xi)<0$ (or, equivalently, $\tau(B)<0$ and $\tau(D)<0$).
\end{enumerate}
Then $Z$ is geometrically ergodic and positive Harris recurrent. Furthermore, the stationary distribution $\pi$ has a finite second moment, i.e. $\int_{K\times\mathbb R^n} \|\xi\|^2 \pi(d\xi) < \infty$.
\end{theorem}
\begin{remark}\label{megasuper}
\begin{enumerate}
\item The inclusion of finite activity and state-dependent jump behavior is possible (similarly to the main results of this paper).
\item For ``canonical state spaces" $\mathbb R_+^m\times\mathbb R^n$, sufficient conditions for irreducibility and aperiodicity (that is {\ref{part1 term}} of Theorem \ref{megageil}) can be can be found in \cite{ZhangGlynn2018}. 
\item The classical term-structure literature (for $K=\mathbb R_+^m$, see \cite[Section 2, (C1)--(C4)]{glasserman2010moment}) uses stronger conditions on drift and diffusion coefficients of $Z$. In particular, the linear
drift coefficients $D$ and $B$ must have real eigenvalues due to econometric identification issues.
\item \label{nacka bazzi} The ergodicity result of \cite[Theorem 4.1]{barczy2014stationarity} is the special case, where $K=\mathbb R_+$, $C=0$ and the constant diffusion matrix $E$ vanishes. They however
prove directly Assumption {\ref{part1 term}}. The condition {\ref{part2 term}} clearly is just implied by strictly negative linear drift coefficients.
\end{enumerate}
\end{remark}
\begin{proof}
Since $\tau(D)<0$, due to \cite[Theorem 2.2.1]{Hornetal1991} there exists $Q\in S_n^{++}$ such that $D^\top Q+Q D \in-S_n^{++}$ (and conversely, this implies $\tau(D)<0$.) 

Since $B$ is qmi, there exists $\eta_0\in K\setminus\{0\}$, such that $B^\top \eta_0=\tau(B)\eta_0$ (see Lemma \ref{lem eval qmi}). We may without loss of generality assume
$\inf_{x\in K, \|x\|=1}\langle \eta_0,x\rangle=1$.

Let $\vartheta_0$ be a positive constant satisfying
\begin{equation}\label{eq: theta bound}
\vartheta_0<\min\left(\vert\tau(D^\top Q+Q D)\vert^3,\frac{(2\vert\tau(B)\vert)^{3/2}}{\|C\|^3\|D\|^3}\right).
\end{equation}
We introduce the quadratic test function $f: K\times \mathbb R_+^n\rightarrow [1,\infty)$ by setting
\[
f(x,y):=1+\langle\eta_0,x\rangle^2+\vartheta_0 \langle y, Qy\rangle ,\quad x\in K,\quad y\in \mathbb R^n
\]
and obtain for any $(x,y)\in K\times\mathbb R_+^n$,
\begin{align*}
&\langle B(x) , f_x\rangle+\langle C(x)+D(y),f_y\rangle\\&\qquad=2\langle x, B^\top \eta_0\rangle\langle \eta_0,x\rangle+2\vartheta_0\langle C(x),Qy\rangle+ \vartheta_0\langle(D^\top Q+QD)y, y\rangle\\ 
&\qquad\leq -2\vert\tau(B)\vert\langle \eta_0,x\rangle^2-\vartheta_0\vert\tau(D^\top Q+Q D)\vert \|y\|^2+ 2\vartheta_0 \|C\|\|Q\|\|x\|\|y\|\\
&\qquad\leq -2\vert\tau(B)\vert\langle \eta_0,x\rangle^2-\vartheta_0\vert\tau(D^\top Q+Q D)\vert \|y\|^2+ 2\vartheta_0 \|C\|\|Q\|\langle \eta_0,x\rangle\|y\|\\
&\qquad\leq -2\vert\tau(B)\vert\langle \eta_0,x\rangle^2-\vartheta_0\vert\tau(D^\top Q+Q D)| \|y\|^2 +2 \vartheta_0^{1/3}\|C\| \|Q\|\langle \eta_0,x\rangle \vartheta_0^{2/3}\|y\|\\
&\qquad\leq -(2\vert\tau(B)\vert-\vartheta_0^{2/3}\|C\|^2\|Q\|^2)\langle \eta_0,x\rangle^2-\vartheta_0(\vert\tau(D^\top Q+QD)\vert-\vartheta_0^{1/3})\|y\|^2\\
&\qquad\leq -\varepsilon (f(x,y)-1)
\end{align*}
for some $\varepsilon>0$, where in the last but one estimate we have used the fact that $\inf_{\|x\|=1}\langle \eta_0,x\rangle=1$ and in the last one
we have used \eqref{eq: theta bound} and the fact that $\langle Qy,y\rangle$ is non-degenerate, therefore bounded from below by a positive multiple of $\|y\|^2$.

The process $Z$ is an affine diffusion, and thus $m$--polynomial for any $m\geq 2$, see \cite{cuchiero2012polynomial}. Hence $f$ lies in the domain
of the extended generator.

Since the infinitesimal generator $\mathcal A$ of $(X,Y)$ is a second order differential operator with affine drift and affine diffusion coefficients, its formal application
yields
\[
\mathcal Af(x,y)\leq -\varepsilon( f(x,y)-1)+O(\|x\|+\|y\|)
\]
and since $(f-1)^{1/2}$ is norm-like, we can find a large enough ball $\mathcal C=\{(x,y)\mid \|x\|^2+\|y\|^2\leq R^2\}$ such that
\[
\mathcal Af(x,y)\leq -\frac\varepsilon2 f(x,y)+d 1_{\mathcal C}(x,y),\quad x\in K,\quad y\in \mathbb R^n.
\]
We conclude that $(\mathcal A,f)$ satisfies the Lyapunov -Foster drift condition.

The rest of the proof follows the lines of the proof of Theorem \ref{prop:FLdrift_expo}.
\end{proof}
\theendnotes
\bibliographystyle{abbrvnat}
\providecommand{\MR}[1]{}

\end{document}